\documentclass[english,11pt]{article}
\usepackage{amsfonts}
\usepackage[T1]{fontenc}
\usepackage{amssymb,amsmath}
\usepackage{mathrsfs}
\usepackage{amsthm}
\usepackage{fancyhdr}
\usepackage{times}
\usepackage{color}
\usepackage[usenames,dvipsnames,svgnames,table]{xcolor}
\usepackage{textcomp}
\usepackage{hyperref}
\usepackage{graphicx}
\usepackage{caption}
\usepackage{subcaption}
\usepackage{float}

\theoremstyle{plain}
\numberwithin{equation}{section}
\newtheorem{theorem}{Theorem}[section]

\theoremstyle{definition}

\theoremstyle{plain}

\newtheorem{corollary}[theorem]{Corollary}

\newtheorem{lemma}[theorem]{Lemma}
\newtheorem{proposition}[theorem]{Proposition}

\theoremstyle{definition}
\newtheorem{remark}[theorem]{Remark}

\newtheorem*{ansatz}{Ansatz}

\newcommand{\E}{\mathbb{E}}        %esperanza
\newcommand{\V}{\mathbb{V}ar}    %variancia
\newcommand{\Cov}{\mathbb{C}ov} 
\newcommand{\NN}{\mathbb{N}}         %conjunto numeros naturales
\newcommand{\Z}{\mathbb{Z}}          %conjunto numeros enteros
\newcommand{\p}{\mathbb{P}}          %probabilidad
\newcommand{\I}{\mathbf{1}}              %indicatriz
\newcommand{\LL}{\mathcal{L}}             % generador

\begin{document}

 \title{Convergence properties of many parallel servers
 	under power-of-$D$ load balancing}
 \author{ M.C. Fittipaldi, M. Jonckheere, S.I. L\'opez}
%\date{\today}

\maketitle
 
\begin{abstract}
We consider a system of $N$ queues with decentralized load balancing such as power-of-$D$ strategies (where $D$ may depend on $N$) and generic scheduling disciplines. To measure the dependence of the queues, we use the \textit{clan of ancestors}, a technique coming from interacting particle systems. Relying in that analysis we prove quantitative estimates on the queues correlations implying propaga-\\tion of chaos for systems with Markovian arrivals and general service time distribution. This  solves the conjecture posed by Bramsom et. al. in \cite{BLP2} concerning the asymptotic independence of the servers in the case of processor sharing policy. We then proceed to prove asymptotic insensitivity in the stationary regime for a wide class of scheduling disciplines and obtain speed of convergence estimates for light tailed service distribution.
\end{abstract}

\section{Introduction}

%Within the theory of stochastic networks, multi-server resource sharing system is an important mathematical model for studying telephone switch systems, digital cellular mobile networks, computer networks and so on.
%During the last two decades, considerable attention has been paid to the study of this kind of systems \cite{A, LXKGLG,S,T}.\\

A central problem in a multi-server resource sharing system  is to decide which server an incoming job will be assigned to.
There are several examples of such systems as call centers, server farms or distributed systems with web applications running on different servers. A good load balancing scheme typically aims to optimize performance metrics like delay, and should be robust to statistical heterogeneity of job sizes.\\

One possible scheme (commonly used in small web server farms) is the \textit{join-the-shortest-queue (JSQ)}, which assigns a new arrival to the server having the least number of unfinished jobs in the system.
Recently, Gupta et al \cite{GBSW} showed that for large  \textit{processor sharing (PS)} service systems, the JSQ scheme is nearly optimal in terms of minimizing the mean sojourn time of jobs while exhibiting low sensitivity to the type of job length distribution.\\
However, the JSQ scheme has a major downside: when applied to a system consisting of a large number of servers, it requires the state information of all the servers in the system to make job assignment decisions. One way to avoid this obstacle is to use dynamic randomized algorithms where the dispatcher selects among a small subset of servers (chosen uniformly at each arrival).
It has been shown that randomized load balancing schemes provides a significant reduction in mean sojourn times associated with JSQ when service times are assumed to be exponentially distributed (see \cite{VDK}).
Indeed, as argued in \cite{M, VDK}, most of the gains in average sojourn time reduction are obtained when selecting 2 servers at random referred to as the \textit{Power-of-Two rule}. This is also referred to as the \textit{JSQ($2$)} scheme.\\
The case when the service time is exponentially distributed was thoroughly studied by Vvedenskaya et al \cite{VDK}. In this case  the evolution of the system is given by a countable state Markov chain where a state is given by the number of jobs at each queue, and there exists a unique equilibrium distribution, that it is exchangeable with respect to the ordering of the queues.\\ 
When the service times are generic, the underlying Markov process will typically have an uncountable state space, and positive Harris recurrence for the process is no longer obvious \cite{B,FC}. 
For JSQ networks with general service times, Bramson et al. \cite{BLP1} described a program for analyzing the limiting behavior of the equilibria as $N \rightarrow \infty$. In this setting, an important step is to show that any fixed number of queues become independent of one another, with each converging to a limiting distribution that is the equilibrium for an associated Markov process with a single queue, sometimes called a \textit{cavity process}. This process corresponds, in an appropriate sense, to ``setting $N = \infty$'' in the JSQ network and viewing the corresponding infinite dimensional process at a single queue. %They refer to this equilibrium as the equilibrium environment.
Although it seems that this independence should hold in a very general setting, including a wide range of service disciplines, demonstrating it can be a difficult technical issue. In \cite{BLP1}, this independence and convergence to the equilibrium environment were stated as an ansatz (see below for a formal definition). This ansatz was demonstrated in Bramson et al. \cite{BLP2} for networks where the service discipline is FIFO and the service distribution has a decreasing hazard rate.\\
In the recent work of \cite{AgRa}, this ansatz was also proven for the FIFO policy case, under the assumption of
bounded hazard rate for the service distribution, using measure valued processes viewed as an interacting
particle system. They also proved uniqueness of the stationary state of the hydrodynamic limit, without proving convergence to the invariant regime.

Here, we aim at proving both the ansatz, the convergence of a tagged queue and convergence under the stationary regime, using completely different techniques and assumptions on the model.
In particular, while we assume Poisson arrivals (which is not assumed in \cite{AgRa}), we do not assume anything on the jobs distribution to prove the ansatz and our results are valid for a much wider class of service disciplines. 
We now describe our contribution in more details.
%In this article, we first address this question for fairly general systems using techniques that were successfully used to prove convergence of empirical measures of particle systems (for instance in \cite{FV}) and allowing to get quantitative estimates on the time correlations between servers and as a consequence, convergence speeds (for finite time intervals) towards the mean-field limit. 
%Our second result concerns the asymptotic insensitivity of systems with symmetric scheduling.

 \subsection*{ Contribution}

 The main goal of this work is to study the asymptotic behavior of the system when the number of servers goes to infinity when the router scheme is JSQ($D$). The service policy is (except when explicitely stated otherwise) a generic work-conserving policy.
 
 % \textcolor{red}{and the service policy is symmetric}. \\

 Let us consider a system with $N$ parallel queues with independent and identical servers having the same service policy and a single dispatcher. Tasks arrive at the dispatcher as a Poisson process of rate $\lambda N$, and are instantaneously forwarded to one of the servers, with $\lambda < 1$. Tasks can be queued at the various servers, but cannot be queued at the dispatcher. The dispatcher assigns immediately an incoming task to a server, which is the one with the shortest queue among $D$ uniformly at random selected servers ($1 \leq D \leq N$ ). The above-described scheme is called the load balancing model with JSQ($D$) allocation $N$ queues, with some given local service policy.\\
 
 Denote by $X_t=X^N_t=(X^N_t(i) )_{i=1}^{N}$ the process where $X_t(i) = X^N_t(i)$ is the number of tasks at server $i$ at time $t$. Note that $X$ takes values in the configuration space $ \Lambda_N := (\mathbb{Z}_+)^N$. For a given configuration $\xi= (\xi(1),...,\xi(N))$ in $\Lambda_N$, we denote by $m(\xi)= (m(\xi))_{k \in \mathbb{Z}_+}$ its induced empirical measure :
 \begin{equation*}
   m_k(\xi)= \dfrac{\sum\limits_{i=1}^N \I_k(\xi(i))}{N}, \qquad k \in \mathbb{Z}_+,
 \end{equation*}
 so $m_k(\xi)$ represents the proportion of the $N$ servers which have exactly $k$ tasks queued in the configuration $\xi$. For the sake of brevity we will denote $m(t)=m(X_t^N)$, $m_k(t)=m_k(X_t^N)$, for $k$ in $\mathbb Z _+$. Also, we denote by $\pi _k ( \xi)$ to the number of servers with a queue length at least equal to $k$, that is
  \begin{equation*}
 \pi_k(\xi)= \sum\limits_{i=1}^N  \I_{ \{ \xi(i) \geq k \} }, \qquad k \in \mathbb{Z}_+.
 \end{equation*}
 Similarly, we abbreviate $\pi(t)=\pi(X_t^N)$, $\pi_k^N(t)=\pi_k(X_t^N)$, for $k$ in $\mathbb Z _+$. 
 Sometimes we will use $\pi^N(t)$, $\pi_k^N(t)$, $m^N(t)$ and $m_k^N(t) $ to emphasize the dependence on $N$.
   
% \subsection{Main results} 
 We are interested in $m(X^N_t)$, the empirical measure of the queueing system, when the number of servers $N$ grows to infinity. In particular, we want to prove the following Ansatz from Bramson {\em et al} \cite{BLP1} a class of policies.
 
 \begin{ansatz}
 \textit{Consider a load balancing system operating under the $JSQ(D)$ allocation, with $D>1$, $\lambda < 1$, and a given local service discipline. The jobs are assumed to have an arbitrary service time distribution with mean 1. Then, in the large $N$ limit, there is a unique equilibrium distribution. Moreover, under this distribution, any finite number of queues are independent.}
 \end{ansatz} 
 
 In \cite[Section 3]{BLP1}, the authors show, under the assumption of the Ansatz, for the PS service discipline and general distribution, that the distribution of the number of flows passing through it is insensitive to the service distribution. Specifically, the asymptotic fraction of the queues with at least $k$ jobs is given by
 \begin{equation} \label{Pk}
  P_k = \lambda^{\frac{D^k -1}{D -1}}, \quad \text{ for } k \leq N \qquad \text{ and } P_k = 0, \quad \text{ for } k > N .
 \end{equation}
Thus, even though important properties such as reversibility and insensitivity do not hold for any finite $N$, they emerge in the limit as $N$ grows to infinite.

 \begin{remark}
  The Ansatz for the case when the service time distribution is exponential was proven by Vvedenskaya {\em et al}. In \cite{VDK}, they found that, for fixed $D$, $D \geq 2$, the limiting probability that the number of jobs in a given queue is at least $k$ is $\lambda^{(D^k-1)/(D-1)}$, as the number of queues $N$ goes to infinity. 
 \end{remark} 
 
 Thanks to the work of Bramson {\em et al.}, we see that the one crucial step to prove the existence of an unique limiting distribution is to establish the asymptotic independence when the number of servers grows.\\

 In this article, we continue the systematic study of decentralized load balancing for systems with generic service time distributions started in \cite{BLP1,BLP2} and more recently in \cite{AgRa}.
 Our contribution is threefold:
 \begin{itemize}
 	\item  We characterize the evolution of the dependencies between the servers using a clan of ancestors construction similar to the one defined for instance in \cite{AFG,FM} for Fleming-Viot processes. This construction is based only on arrivals. Therefore, we obtain correlations bounds of the empirical measure of the servers, independently of the specific service distribution. This in turn allows to prove convergence of the empirical measure for fixed time intervals, as well as convergence speed estimates. These results are given in Section \ref{sec:chaos}.
 	
    \item We then study the asymptotic behavior of a tagged server in  the queueing model. In fact, we prove the existence of the limiting distribution for the tagged server as the number of servers grows to infinity. By applying the correlations bounds obtained before, we prove that the parameters of this distribution do not depend on the state of the whole queue system. This is done in Section \ref{sec:queue}. 
 	
 	\item Finally, in Section \ref{sec:stat} building on the previous results, we show the convergence of the empirical measures of the system under the stationary regime and we rigorously prove the asymptotic insensitivity to the service time distributions for symmetric policies.
(This result was announced without a complete proof in several previous articles \cite{BLP2, GBSW}.) 	
 	 Under the assumption of having an exponential moment for the service time distribution we prove exponentially fast convergence for a tagged queue with FIFO or symmetric scheduling.
 	
 %	for a very general class of local symmetric scheduling, which includes processor sharing, FIFO, LIFO and Round Robin as examples. 
\end{itemize}

\section{Propagation of Chaos}\label{sec:chaos}

In this section we prove the asymptotic independence of two entries of the empirical measure of the system, when the initial configuration of the $N$ parallel queues is deterministic and arbitrary. Explicitly, the main result is the following:

\begin{proposition}\label{propagationofchaos}
	For $t \geq 0$ and $k, l$ in $\mathbb \mathbb Z_+$ 
	\begin{equation}\label{bound1/N}
	\sup\limits_{\xi \in \Lambda_N} \Big| \mathbb{E}^N_{\xi} \left[m_k(t)m_l(t)\right] - \mathbb{E}^N_{\xi} [m_k(t)] \mathbb{E}^N_{\xi}[m_l(t)] \Big| \leq \dfrac{1}{N} + \dfrac{2(3/2)^DN}{(N-D)} \left[ N\ln \left(\frac{N+u_N(t)-1}{N}\right) - \dfrac{(N-1)(u_N(t)-1)}{(N+u_N(t)-1)}\right],
	\end{equation}
	where $u_N(t)= \exp{\left( \dfrac{2^{D-1}\lambda DN}{N-D}t\right)}$. 
\end{proposition}

\begin{remark}
Note that $D$ could depend on $N$.
\end{remark}

As a consequence, we have the following corollary:
	\begin{corollary}\label{velocity}
			Suppose $D$ does {\bf not} depend on $N$.
			For $t \geq 0$ and for all $k, l$ in $\mathbb \mathbb Z_+$, there exists $N_0$ such that for $N \ge N_0$: 
			\begin{equation}\label{bound1/N}
			\sup\limits_{\xi \in \Lambda_N} \Big| \mathbb{E}_{\xi} \left[m^N_k(t)m^N_l(t)\right] - \mathbb{E}_{\xi} [m^N_k(t)] \mathbb{E}{\xi}[m^N_l(t)] \Big| \leq   \dfrac{1 +  (3/2)^D\left(e^{2^D\lambda Dt} -1\right)}{N} 
			\end{equation}
	\end{corollary}

The strategy to prove Proposition \ref{propagationofchaos}, is to show the existence of upper bounds on the correlations of two entries of the empirical measure that hold uniformly over all the possible initial configuration of the particles.

% Our proof can be carried on in the case when we have an arbitrary finite number of entries, so any finite number of entries evolve independently in the limit when $N$ grows to infinity.
% \matt{No entendi eso...}

 \subsection{Ancestors clans} 
 Let us denote by $\omega^{A}$ the time-realizations of the arrival process and by $\zeta_t \subseteq \{1,...,N\}$ the set of $d$ randomly chosen servers to potentially attend a task appearing at an arrival time $t$ of the process $A$. \\

 For each server $i$ and each time $t$ we build simultaneously a set of labels $\psi^i(t) \subseteq \{1,...,N\}$ which represents the servers which could potentially influence $X^N_t(i)$. These sets are constructed using the ``backward" trajectories $\omega^{A}[-t, 0]$ and the sets $\zeta_s$, indexed by the arrival times in $[-t,0]$ of these trajectories. The process $t \mapsto \psi^i(t)$ may only change at the arrival times of $\omega^{A}$. We describe explicitly the evolution of the process $\psi^i(t)$ in the following. \\

 Fix $i$ in $\{1,...,N\}$. Let us denote by $\omega_{\psi^i}^A$ the thinning from the process $\omega^A$, with the following rule: we keep an arrival time $t$ if and only if $\psi_t^i \cap \zeta_t \neq 0$. Let $-v$ be the largest arrival time of $\omega^A_{i} [-t, 0)$ and $\zeta_v$ its associated set. For $0 \leq s < v$ we set $\psi^i(s) := \{i\}$ and $\psi^i(v) := \zeta_v $. Assume that for $s < t$, $\psi^i(s)$ is defined. Let $-u$ be the largest arrival time of $\omega^A_{\psi^i(s)}[-t, -s)$ and $\zeta_u$ its associated set. Then, 
 $$\forall r \in [s, v),\, \psi^i (r) = \psi^i (s) \quad \text{ and }  \quad \psi^i(u) = \psi^i (s) \cup \zeta_u .$$
 Note that for positive $t$, $\psi^i(t)$ is $\sigma(\omega^A [-t, 0))$-measurable, and that for any set of labels $C \subseteq \{1, . . . , N\}$, we have 
 $$\{\psi^i(t) = C\} \in \sigma(\omega^A [-t, 0)) \quad \text{ and } \quad \{\psi^i(t) = C, X_t(i) = k\} \in \sigma(\omega[-t, 0)).$$

 The next lemma gives an upper bound of order $\frac{1}{N}$ to the probability of having two intersecting sets of labels associated to different servers.

 \begin{lemma}\label{cota}
	For $i$, $j$ distinct servers, and $t > 0$,
	\begin{equation*}
	\mathbb{P}\left[ \psi^i(t) \cap \psi^j(t) \neq  \emptyset\right] \leq \dfrac{(3/2)^DN}{(N-D)} \left[ N\ln \left(\frac{N+u_N(t)-1}{N}\right) - \dfrac{(N-1)(u_N(t)-1)}{(N+u_N(t)-1)}\right], 
	\end{equation*} 
	where 
	\begin{equation}\label{udef}
	u_N(t)= \exp{\left( \dfrac{2^{D-1}\lambda DN}{N-D}t\right) }.
	\end{equation}
 \end{lemma}

\begin{remark}
	As
	\[
	\lim \limits_{N \rightarrow \infty} N \ln \left(\frac{N+u_N(t)-1}{N}\right)=
	\ln  \left( 	\lim \limits_{N \rightarrow \infty}  \left[   1 + \frac{ u_N(t)-1 }{ N}  \right]^N  \right)= \exp{\left( 2^{D-1}\lambda D t\right) } -1
	\]
	and 
	\[
	\lim\limits_{N \rightarrow \infty} \dfrac{(N-1)(u_N(t)-1)}{N+u_N(t)-1}= \exp{\left( 2^{D-1}\lambda D t\right) } -1,
	\]	
	we have that
		\[
	\lim\limits_{N \rightarrow \infty} \dfrac{(3/2)^D N}{N-D} \left[ N\ln \left(\frac{N+u_N(t)-1}{N}\right) - \dfrac{(N-1)(u_N(t)-1)}{N+u_N(t)-1}\right]=0,  
	\]
		for every fixed $t > 0$.
\end{remark}

 \textbf{ \textit{Proof of Lemma \ref{cota}} : }
 	Here we prove the case $D=2$. The general case is slightly more involved and the details can be found in the Appendix \ref{approp}. \\

		First, we bound the expected number of members of the clan $\psi^i(t)$. We observe that	
 \begin{eqnarray}\label{clanbound}
  	\dfrac{d\mathbb{E}\left[|\psi^i(t)|\big|\sigma(\omega[-t,0))\right]}{dt}  
  &=&  \lambda N  \mathbb{P}[\text{choose exactly one server in } \psi^i(t)  ]  \nonumber \\
&  =&	\lambda N \frac{  |\psi^i(t)| ( N- |\psi^i(t)|)  }{ { N  \choose 2 }} 
\nonumber  \\
& = & \frac{  2 \lambda  N   }{ N-1} |\psi^i(t)| \Big( 1- \frac{|\psi^i(t)|}{N}  \Big) \nonumber  \\
& \leq &  \dfrac{2^{D-1}\lambda DN}{N-D}|\psi^i(t)|\left( 1  - \dfrac{|\psi^i(t)|}{N} \right).
\end{eqnarray}
 
 	Hence, using Jensen inequality 
 	\[
 	\dfrac{d\mathbb{E}[|\psi^i(t)|]}{dt} \leq  \dfrac{2^{D-1}\lambda DN}{N-D} \mathbb{E}[|\psi^i(t)|]\left(1 - \frac{\mathbb{E}[|\psi^i(t)|]}{N}\right) ,
 	\]
 	which corresponds to the logistic differential equation, so
 	\begin{equation}\label{cota1}
 	\begin{split}
 	\mathbb{E}[|\psi^i(t)|] & \leq \dfrac{N \exp{\left(\dfrac{2^{D-1}\lambda DN}{N-D}t \right)}}{N  + \exp{\left( \dfrac{2^{D-1}\lambda DN}{N-D}t\right) }-1}.
 	\end{split}
 	\end{equation} 		
 	Second, we show that for $i, j$ two distinct labels,
 	\begin{equation}\label{cota2}
 	\mathbb{P}\left[ \psi^i(t) \cap \psi^j(t) \neq \emptyset\right] \leq\dfrac{3^D\lambda N}{(N-D)^2} \int_0^t \mathbb{E}[|\psi^i(s)|]\mathbb{E}[|\psi^j(s)|]ds.
 	\end{equation} 
 	For this purpose, note that
 	\begin{equation*}
 	\mathbb{P}\left[ \psi^i(t) \cap \psi^j(t) \neq \emptyset \right] = \int_0^t \mathbb{E}\left[\dfrac{d\mathbb{P}(\psi^i(s) \cap \psi^j(s) \neq \emptyset \big| \sigma(\omega[-s, 0)))}{ds}\right]ds
 	\end{equation*}
 	and 
 	\begin{align*}
    \dfrac{d\mathbb{P}(\psi^i(s) \cap \psi^j(s) \neq \emptyset \big| \sigma(\omega[-s, 0)))}{ds} = & \lambda N \mathbb{P}\left[\psi^i(s) \cap \psi^j(s) \neq \emptyset \big|\psi^i(s^-) \cap \psi^j(s^-) = \emptyset \right]\\
 	=& \lambda N \frac{ |\psi^i(s^-)|  |\psi^j(s^-)|   }{ {N \choose 2} }  \textbf{1}_{ \{ \psi^i(s^-)\cap \psi^j(s^-)=\emptyset \} } \\
 	=& \frac{ 2 \lambda }{ N-1 }  |\psi^i(s^-)|  |\psi^j(s^-)| \textbf{1}_{  \{ \psi^i(s^-)\cap \psi^j(s^-)=\emptyset  \}} \\
 	 \leq &  \frac{  3^D  \lambda N }{ (N-1)^2 } |\psi^i(s^-)|  |\psi^j(s^-)|   \textbf{1}_{ \{ \psi^i(s^-)\cap \psi^j(s^-)=\emptyset \} }.
  	\end{align*}	
 
 	Then
 	\begin{align*}
 	\mathbb{E}\left[ \dfrac{d\mathbb{P}(\psi^i(s) \cap \psi^j(s) \neq \emptyset \big| \sigma(\omega[-s, 0)))}{ds}\right]
 &	\leq \dfrac{3^D\lambda N}{(N-D)^2} \mathbb{E}\left[|\psi^i(s)||\psi^j(s)|\textbf{1}_{\{\psi^i(s)\cap \psi^j(s)=\emptyset\}} \right]\\
 	& \leq \dfrac{3^D\lambda N}{(N-D)^2}  \, \sum\limits_{B\cap C  = \emptyset} |B||C| \mathbb{P}[\psi^i(s)=B] \mathbb{P}[\psi^j(s)=C] \\
 	&  \leq \dfrac{3^D\lambda N}{(N-D)^2} \,\mathbb{E}[|\psi^i(s)|]\mathbb{E}[|\psi^j(s)|],
 	\end{align*} 
 	where in the last line we used that for two non-overlapping subsets of labels $B$ and $C$, $\{\psi^i(t) = B\}$ and  $\{\psi^j (t) = C\}$ are independent. This concludes the proof of \eqref{cota2}. Now, we use \eqref{cota1} and \eqref{cota2} to obtain:           
 	\begin{equation}\label{integral}
 	\begin{aligned}
 	\mathbb{P}(\psi^i(s) \cap \psi^j(s) \neq \emptyset)&\leq \dfrac{3^D\lambda N}{(N-D)^2}\int_0^t\mathbb{E}[|\psi^i(s)|]\mathbb{E}[|\psi^j(s)|]ds \\
 	&\leq \dfrac{3^D\lambda N}{(N-D)^2}\int_0^t\dfrac{N^2 \exp{\left(  2^{D}\dfrac{\lambda DNs}{N-D} \right)}}{\left(N  + \exp{\left( 2^{D-1}\dfrac{\lambda DNs}{N-D}\right) }-1\right)^2}ds
 	\end{aligned}
 	\end{equation}
 	Using the definition of $u_N(t)$ \eqref{udef}, we can rewrite 
 	\begin{align*}
 	\dfrac{3^D\lambda N}{(N-D)^2}\int_0^t\dfrac{N^2 u^2(s)}{\left(N  + u(s)-1\right)^2}ds 
 	& = \dfrac{2(3/2)^D}{D(N-D)}\int_1^{u_N(t)}\dfrac{N^2u}{\left(N  + u-1\right)^2}du \\
 	& = (3/2)^D\dfrac{2N^2}{D(N-D)}\int_N^{N+u_N(t)-1}\dfrac{v-(N-1)}{v^2}dv \\
 %	& = (3/2)^D\dfrac{2N^2}{D(N-D)} \left[ \ln v \bigg|_{N}^{N+u_N(t)-1} + \dfrac{N-1}{v}\bigg|_{N}^{N+u_N(t)-1}\right]\\
 	& = (3/2)^D\dfrac{2N}{D(N-D)} \left[ N\ln \left(\frac{N+u_N(t)-1}{N}\right) - \dfrac{(N-1)(u_N(t)-1)}{(N+u_N(t)-1)}\right]
 	\end{align*}
 and conclude the desired result. 
 	\begin{flushright}
 	 $\square$	
 	\end{flushright}
	
%\textcolor{red}{We have
%	\begin{align*}
%	\dfrac{3^D\lambda N}{(N-D)^2}\int_0^t\dfrac{N^2 u^2(s)}{\left(N  + u(s)-1\right)^2}ds  & = \dfrac{2(3/2)^D}{D(N-D)}\int_1^{u_N(t)}\dfrac{N^2u}{\left(N  + u-1\right)^2}du \leq \dfrac{2(3/2)^D}{D(N-D)}\int_1^{u_N(t)}u du \\ & \leq \dfrac{2(3/2)^D}{D(N-D)} \left[ \dfrac{u^2}{2} \bigg|_{1}^{u_N(t)}\right] = \dfrac{(3/2)^D\left(u(2t)-1\right)}{D(N-D)}
%	\end{align*}}

 \subsection{Proof of the Proposition \ref{propagationofchaos}} 
By definition, we have that
\begin{equation*}
\begin{split}
\big| \E _{\xi} [m_k^N(t)m_l^N(t)] - \E_{\xi} [m_k^N(t)]\E_{\xi} [m_l^N(t)]\big| 
\leq \dfrac{1}{N^2} \sum\limits_{i=1}^N\sum\limits_{j=1}^N \big| \mathbb{P}[X_t (i) = k, X_t(j) = l] - \mathbb{P}[X_t (i) = k] \mathbb{P}[X_t(j) = l]\big|\\
\leq \dfrac{1}{N^2} \sum\limits_{i=1}^N\sum\limits_{j=1, j\neq i}^N \big| \mathbb{P}[X_t (i) = k, X_t(j) = l] - \mathbb{P}[X_t (i) = k] \mathbb{P}[X_t(j) = l]\big| + \dfrac{1}{N}.
\end{split}
\end{equation*}
 Then it will be enough to show that for any pair of numbers $k, l$, any time $t \geq 0$, and initial configuration $\xi$,

 \begin{equation}\label{covariance1}
 \begin{split}
  \sum\limits_{i =1}^N \sum\limits_{ j=1, j\neq i }^{N} &\big|\mathbb{P}[X_t (i) = k, X_t(j) = l] -\mathbb{P}[X_t (i) = k] \mathbb{P}[X_t(j) = l]\big|\\
  & \leq \dfrac{2(3/2)^DN^2(N-1)}{(N-D)} \left[ N\ln \left(\frac{N+u_N(t)-1}{N}\right) - \dfrac{(N-1)(u_N(t)-1)}{(N+u_N(t)-1)}\right].
 \end{split}
 \end{equation}

 For a subset $B$, $\{X_t(i) = k, \psi^i(t) = B\}$ is $\sigma(\omega_B [-t, 0))$-measurable, where $\omega_B $ is the thinning from the arrival process $\omega^A$ (we keep an arrival time $t$ if and only if some $i$ in $B$ belongs to the set  $\zeta_t$). Thus, for two non-overlapping subsets of labels $B$ and $C$, it happens
 $$\mathbb{P}[ \psi^i (t) = B, \psi^j (t) = C, X_t(i) = k, X_t (j) = l] = \mathbb{P}[\psi^i(t) = B, X_t (i) = k] \mathbb{P}[\psi^j (t) = C, X_t(j) = l] .$$

 Then, for $i \neq j$, we can write 
 \begin{equation}\label{together}
  \begin{split}
   \mathbb{P}[X_t(i) = k, X_t (j) = l] =\, & \mathbb{P}\left[\psi^i(t) \cap \psi^j (t) \neq \emptyset, X_t (i) = k, X_t (j) = l\right]\\
   & + \sum\limits_{B\cap C = \emptyset} \mathbb{P}\left[ \psi^ i (t) = B, \psi^j (t) = C, X_t (i) = k, X_t (j) = l \right]\\
   = \, & \mathbb P \left[\psi^i(t) \cap \psi^j (t) \neq \emptyset, X_t (i) = k, X_t (j) = l\right]\\ 
   & + \sum\limits_{B \cap C = \emptyset}\mathbb{P}\left[ \psi^i (t) = B, X_t (i) = k\right] \mathbb{P} \left[ \psi^j (t) = C, X_t(j) = l\right].
  \end{split}
 \end{equation}
 To compute $\mathbb{P}[X_t (i) = k]\mathbb{P}[X_t(j) = l]$, we can use an independent marked-point processes with the same evolution as $X$ (put a tilda for the independent copy). Then we have a similar decomposition to \eqref{together}:
 \begin{equation}\label{apart}
  \begin{split}
   \mathbb{P}[X_t(i) = k] \mathbb{P}[X_t(j) = l] =\, & \mathbb{P}\left[\psi^i (t) \cap \tilde{\psi}^j (t) \neq \emptyset, X_t(i) = k, \tilde{X}_t (j) = l\right]\\
   & + \sum\limits_{B \cap C = \emptyset}\mathbb{P}\left[ \psi^i(t) = B, \tilde{\psi}^j (t) = C, X_t (i) = k, \tilde{X}_t(j) = l\right]\\
   =\, &  \mathbb{P}\left[\psi^i(t) \cap \tilde{\psi}^j(t) \neq \emptyset, X_t(i) = k, \tilde{X}_t(j) = l\right]\\
   & + \sum\limits_{B \cap C = \emptyset} \mathbb{P} \left[\psi^i(t) = B, X_t (i) = k\right] \mathbb{P}\left[ \psi^j(t) = C, X_t (j) = l\right] .
  \end{split}
 \end{equation}

 Subtracting \eqref{together} and \eqref{apart} we get that for $j \neq i$,
 \begin{equation}\label{diff}
  \begin{split}
   & \big|\mathbb{P}[X_t (i) = k, X_t(j) = l] -\mathbb{P}[X_t (i) = k] \mathbb{P}[X_t(j) = l]\big| \\
   & \leq \mathbb{P}\left[\psi^i(t) \cap \psi^j (t) \neq \emptyset, X_t (i) = k, X_t (j) = l\right] + \mathbb{P}\left[\psi^i(t) \cap \tilde{\psi}^j(t) \neq \emptyset, X_t(i) = k, \tilde{X}_t(j) = l\right]\\
   & \leq  \mathbb{P}\left[\psi^i(t) \cap \psi^j (t) \neq \emptyset\right] + \mathbb{P}\left[\psi^i(t) \cap \tilde{\psi}^j(t) \neq \emptyset\right]\\
   & \leq \dfrac{2(3/2)^DN}{(N-D)} \left[ N\ln \left(\frac{N+u_N(t)-1}{N}\right) - \dfrac{(N-1)(u_N(t)-1)}{(N+u_N(t)-1)}\right],
  \end{split}
 \end{equation}
 where in the last inequality we used Lemma \ref{cota} to bound each summand (we have used that the result holds also for $\psi^i(t) \cap \tilde{\psi}^j(t)$). Using the last bound and summing over $i$ and $j$ different, we can deduce \eqref{covariance1}:

 \begin{equation*}
  \begin{split}
   \big|\mathbb{E}_{\xi} \left[m_k(t)m_l(t)\right] - \mathbb{E}_{\xi} [m_k(t)]\mathbb{E}_{\xi} [m_l( t)]\big| 
   & \leq {1 \over N^2} \sum\limits_{i,j} \big|\mathbb{P}[X_t(i)=k,X_t(j)=l] - \mathbb{P}[X_t(i)=k] \mathbb{P}[X_t(j)=l]\big| + \dfrac{1}{N} \\
   &  \leq \dfrac{2(3/2)^D(N-1)}{(N-D)} \left[ N\ln \left(\frac{N+u_N(t)-1}{N}\right) - \dfrac{(N-1)(u_N(t)-1)}{(N+u_N(t)-1)}\right]+ \dfrac{1}{N}.
  \end{split}
 \end{equation*}

\begin{flushright}
$\square$
\end{flushright}

We have an analogous  result for the (non-normalized) tail $\{ \pi^N_k (t) \}$, which will be useful in the next section.
\begin{proposition}\label{cotapi} For  any $\xi \in \Lambda_N$, $t\geq 0$ and $k,l \in \mathbb{Z}_+$, we have 
		\begin{equation}\label{covariancepi}
\big| \mathbb{E}_{\xi} \left[\pi^N_k(t)\pi^N_l(t)\right] - \mathbb{E}_{\xi} [\pi^N_k(t)] \mathbb{E}_{\xi}[\pi^N_l(t)] \big| \leq \dfrac{2(3/2)^DN^2(N-1)}{(N-D)} \left[ N\ln \left(\frac{N+u_N(t)-1}{N}\right) - \dfrac{(N-1)(u_N(t)-1)}{(N+u_N(t)-1)}\right] + N,
		\end{equation}
\end{proposition}
\begin{proof}
	By definition of $\pi^N_k(t)$, we have that
\begin{equation*}
\begin{split}
 \big| \mathbb{E}_{\xi} \left[\pi^N_k(t)\pi^N_l(t)\right] - \mathbb{E}_{\xi} [\pi^N_k(t)] \mathbb{E}_{\xi} [\pi^N_l(t)] \big| 
 \leq \sum\limits_{i =1}^N \sum\limits_{ j=1}^{N} \big|\mathbb{P}[X_t (i) \geq k, X_t(j) \geq l] -\mathbb{P}[X_t (i) \geq k] \mathbb{P}[X_t(j) \geq l]\big|\\
 \leq \sum\limits_{i =1}^N \sum\limits_{j=1, j\neq i}^{N} \big|\mathbb{P}[X_t (i) \geq k, X_t(j) \geq l] -\mathbb{P}[X_t (i) \geq k] \mathbb{P}[X_t(j) \geq l]\big| + N
\end{split}
\end{equation*}
	We can adapt the same argument as in the proof of Proposition \ref{propagationofchaos}, to conclude that for any pair of servers $i,j$, $i \neq j$ 
	\begin{equation*}
	 \big|\mathbb{P}[X_t (i) \geq k, X_t(j) \geq l] -\mathbb{P}[X_t (i) \geq k] \mathbb{P}[X_t(j) \geq l]\big| \leq
	 \dfrac{2(3/2)^DN}{(N-D)} \left[ N\ln \left(\frac{N+u_N(t)-1}{N}\right) - \dfrac{(N-1)(u_N(t)-1)}{(N+u_N(t)-1)}\right] 
	\end{equation*}
	and the proof is complete.
\end{proof}

\subsection{Proof of Corollary \ref{velocity}}
From \eqref{cota1}, we have that $\dfrac{Nu_N(t)}{N+u_N(t)-1}$ corresponds to the logistic function, so the natural upper bound for this function is the exponential $u_N(t)=\exp{\left( \dfrac{2^{D-1}\lambda DNt}{N-D}\right) }$. Using this bound in the inequality \eqref{integral}, we have that
	\begin{align*}
		\mathbb{P}(\psi^i(s) \cap \psi^j(s) \neq \emptyset)&\leq 
		 \dfrac{3^D\lambda N}{(N-D)^2}\int_0^t u^2_N(s)ds = \dfrac{3^D\lambda N}{(N-D)^2}\int_0^t \exp{\left( \dfrac{2^D\lambda DNs}{N-D}\right) }ds\\
		 & \leq \left(\dfrac{3}{2}\right)^D \dfrac{\exp{\left( \dfrac{2^D\lambda DNt}{N-D}\right) } -1}{(N-D)D} 
		 \sim \left(\dfrac{3}{2}\right)^D \dfrac{\left(e^{2^D\lambda Dt} -1\right)}{ND} 
	\end{align*} 
	from an $N_0$ onward, and from here
	\begin{align*}
	\big|\mathbb{E}_{\xi} \left[m_k(t)m_l(t)\right] - \mathbb{E}_{\xi} [m_k(t)]\mathbb{E}_{\xi} [m_l( t)]\big| \leq
	\left(\dfrac{3}{2}\right)^D \dfrac{2\left(e^{2^D\lambda Dt} -1\right)}{ND}  + \dfrac{1}{N} \leq \dfrac{1 +  (3/2)^D\left(e^{2^D\lambda Dt} -1\right)}{N} 
	\end{align*}
	
	\begin{flushright}
		$\square$
	\end{flushright}

\section{Convergence of the tagged queue}\label{sec:queue}

In this section we will study the asymptotic behavior a fixed queue. The fixed queue will be called the "tagged" queue, the terminology comes from particle systems. By the symmetry of the model we can fix the first server, so in the following
 the process $X^N(1)$ will be the tagged queue. We will prove that, as $N$ grows to infinity, the rate of arrival to this tagged queue converges a.s. to some constant (depending on time and state of the queue). This will be accomplished by constructing an appropriate coupling.

%, which will be successful under the following mild assumption on the local service policy.

%\begin{hyp}\label{hyH}
 %Let $X^A(t)$ and $X^B(t)$ two queue length processes associated to the same service scheduling but different arrival point processes $A(t)$, $B(t)$ such that $A(t)$ is stochastically dominated by $B(t)$. We say that a service policy satisfies Hypothesis \ref{hyH} if there exists a coupling $( \hat {X}^A, \hat{X}^B)$ of the processes $X^A$ and $X^B$ such that  $\hat {X}^A$ is stochastically dominated by $\hat {X}^B$.	
%\end{hyp}

%Hypothesis \ref{hyH} is satisfied by Processor sharing, FIFO, LIFO and Round robin policies, but it is not satisfied by Random order of service policy (see the Appendix for a proof). As a direct consequence we will conclude convergence in total variation of the law of the tagged queue to the law of some queue which is independent of the system. \\

First, for $x=(x_1,...,x_N) \in \mathbb{Z}_+^N$ and $k \in \mathbb N$ define the deterministic function $\lambda^N :  \mathbb Z _+ ^N  \times \mathbb Z_ + \rightarrow \mathbb [0, \infty)$ by
\begin{equation}\label{lambdaNdef}
\lambda^N (x,k):= 1_{ \{x(1)=k \} } \, \lambda N  \sum\limits_{i=1}^D \frac{1}{i}\frac{ { \pi_k^N(x)-\pi_{k+1}^N(x) -1    \choose i-1 }{\pi_{k+1}^N(x)\choose D-i }}{ {N  \choose D}  } , 
\end{equation}
(recall that $\pi^N_k(x)= \sum_{i=1}^{N} \I_{\{x_i \geq k\}}$ for each $k=0,1,2,...$). In the rest of this section, we will always assume that $\pi^N_k(x) -\pi^N_{k+1}(x)$,  $\pi^N_k(x) \geq D$ since we are only interested in the behavior of $ \lambda^N (x,k)$ as the size $N$ of the parallel system grows to infinite. \\

The interpretation of $\lambda^N (x,k)$ is the following:  Given a system of $N$ servers with queue lengths $x$, when the first server has length $k$ (that is $x(1)=k$) and there is a total rate $\lambda N$ entering the system, the effective arrival rate associated to the first queue is given by $\lambda^N (x,k)$. Each summand in the expression corresponds to the probability of selecting $i$ servers with queue length equal to $k$ and $D-i$ servers with queue length strictly bigger than $k$ and then choosing the server number $1$ between those $D$ servers to allocate an incoming task. \\

For natural numbers $D \leq a<b$ define the function
\begin{eqnarray}\label{eqsum}
	S^D(a,b) & := &(b-1)...(b-(D-1)) + a(b-2)...(b-(D-1))+ a(a-1)(b-3)...(b-(D-1)) \nonumber \\
	& + &  ... \, + a(a-1)...(a-(D-3))(b-(D-1)) + a(a-1)...(a-(D-2))   \nonumber \\
	& = &  \sum_{i=0}^{D-1} (a-0)(a-1)  \dots (a-(i-1)) (a-i)^0 (b-(i+1)) \dots (b-(D-1))
\end{eqnarray}

The next result allows us to easily handle the value of $\lambda^N (x,k)$. Its proof can be found in the Appendix \ref{apconv}.
\begin{lemma}\label{lambdasum}
$$  \lambda^N (x,k)= 1_{ \{x(1)=k \} } (\lambda N) \frac{(N-D)!}{N!} S^D(\pi^N_{k+1}(x), \pi^N_k(x)) .$$
\end{lemma}

A direct corollary of Lemma \ref{lambdasum} is that  $\lambda^N (x,k)$ is uniformly bounded on $N,x$ and $k$:
\begin{eqnarray}\label{lambdaNbound}
 \lambda^{N}(x,k) &=&    1_{ \{x(1)=k \} } (\lambda N) \frac{(N-D)!}{N!} S^D(\pi^N_{k+1}(x), \pi^N_k(x)) \nonumber \\
& \leq &   \lambda    \Big( \frac{N}{N-D+1}\Big) \Big( \frac{N}{N-D+2}\Big) \cdots
\Big( \frac{N}{N-1}\Big) \nonumber  \\
& \cdot &  \sum_{i=0}^{D-1}  \Big( \frac{ \pi_{k+1}^N(x)}{N} \Big)  \Big( \frac{ \pi_{k+1}^N(x) -1}{N} \Big) \cdots
\Big( \frac{ \pi_{k+1}^N(x)- (i-1)}{N} \Big)  \nonumber  \\
& \cdot &   \Big( \frac{ \pi_k^N(x)- (i+1)}{N} \Big)  \Big( \frac{ \pi_k^N(x)- (i+2)}{N} \Big) \cdots  \Big( \frac{ \pi_k^N(x)- (D-1)}{N} \Big) \nonumber  \\
&\leq &  \lambda \Big( \frac{D}{D-D+1}\Big) \Big( \frac{D}{D-D+2}\Big) \cdots
\Big( \frac{D}{D-1}\Big)  \sum_{i=0}^{D-1}  1 ^i   1^{D-1-i} = C_D \lambda  ,
\end{eqnarray}
where $C_D$ is a constant that only depends on $D$.

Using the function $\lambda^N$, we can express the state dependent arrival rate for the tagged queue $X^N(1)$ on the event $\{X^N_t(1)=k\}$ by 
$$ \lambda_t^{k,N} := \lambda^N (X_t^N, k) .$$

A key result, concerning the convergence of the arrival rates for the tagged server, is the following:
\begin{theorem}\label{convergence}
 For any fixed time $t \in [0,T]$, the sequence of random variables $\{ \lambda_t^{k,N}  \}_{N \geq 1}$ converges  in $L^2$ to a constant $\lambda_t ^{k} $. Moreover, there exists a coupling between the random variables $\{ \lambda_t^{k,N}  \}_{N \geq 1}$ such that we have almost surely convergence.
\end{theorem}

\begin{proof}

First, we prove that $\left\{ \lambda_t^{k,N}  \right\}_{N \geq 1}$ converge almost surely to a random variable $\lambda_t^{k} $ on $t \in [0,T]$.  \\

To prove the a.s. convergence of the rates,  we will construct a coupling with the purpose on having an exact algebraic relationship between the rates of arrival to the systems of servers $X^N(1)$ and $X^{N+1}(1)$, given in equation \eqref{lambdaN+1} . \\

Define three independent Poisson process: 
\begin{itemize}
	\item The \textcolor{yellow}{yellow process} with rate $\lambda  N -  (D-1) \lambda$. When there is an arrival correspondent to this process, $D$ servers from the first $N$ are chosen uniformly at random, and the task joins the shortest queue from those servers.
	\item The \textcolor{red}{red process} with rate $(D-1) \lambda$. When there is an arrival correspondent to this process, its task is allocated in a similar manner as in the yellow process.
	\item The  \textcolor{blue}{blue process} with rate $\lambda D$. For this process, the server $N+1$ is always chosen together with $D-1$ uniformly random chosen servers from the first $N$, and the task is allocated in the server with the shortest queue among these.
\end{itemize}
Using these color processes, we couple the systems $X^N$ and $X^{N+1}$: we use the sum of the \textcolor{yellow}{yellow} and \textcolor{red}{red} arrivals to be the arrival process of the system $X^N$, and the sum of the \textcolor{yellow}{yellow} and the \textcolor{blue}{blue} one to be the arrival process of $X^{N+1}$.  Note that the original rate of the arrival process of system $X^{N+1}$ can be recovered as the sum of the rates of the \textcolor{yellow}{yellow} and the \textcolor{blue}{blue} processes:
$$ \lambda (N+1) =  \lambda (N+1)  \frac{  \binom{N}{D}  }{  \binom{N+1}{D}  }  
+ \lambda (N+1) \frac{  \binom{N}{D-1}  }{  \binom{N+1}{D}  } 
=  (\lambda  N -  (D-1) \lambda) + \lambda D  ,$$ 
where the first term in the r.h.s  corresponds to the rate of those arrivals where the $D$ servers are chosen from the first $N$ queues, and the second one to the rate of arrivals when the $N+1$ server is chosen together with $D-1$ other servers. Similarly, the sum of the rates for the \textcolor{yellow}{yellow} and \textcolor{red}{red} processes gives the total rate for the $X^{N}$ system. For all the queueing systems, the initial configurations will be the same in the first $N$ servers, as well as the service times (whenever the service time is associated to a common arrival in both systems).\\

First, we compare directly the rates $\lambda^N(x,k)$ and $\lambda^{N+1}(x,k)$ for a fixed state $x=(x_i)_{i=1}^{\infty} \in (\mathbb{ Z}_{+})^{\mathbb N}$ of the system such that $x(1)=k$.  The main idea is to use the coupling defined above to note that the difference between $\lambda^N(x,k)$  and $\lambda^{N+1}(x,k)$ is given by the difference of the effective arrival rates to the server $1$ corresponding to the \textcolor{red}{red} and the \textcolor{blue}{blue} processes. \\

We can write explicitly the rate $\lambda^{N+1}(x,k)$ in terms of $\pi ^N (x)$ and $x(N+1)$:

\begin{eqnarray}\label{lambdaN+1}
\lambda^{N+1}(x,k) &=& \I_{ \{ x(1)=k  \} } (\lambda N- (D-1) \lambda ) \frac{(N-D)!}{N!} S^D(\pi^N_{k+1}(x), \pi^N_k(x) )   \nonumber \\
 & + & \lambda \I_{ \{ x(1)=k, x(N+1) > k \} }  (\lambda D )  \sum \limits_{i=1}^D \frac{1}{i} \frac{ { \pi_k^N(x)-\pi_{k+1}^N(x) -1     \choose i-1 }{\pi_{k+1}^N(x) \choose D-1-i }}{ {N  \choose D-1}  }  \nonumber \\
 & + & \lambda \I_{ \{  x(1)=k, x(N+1) = k  \} } (\lambda D)  \sum \limits_{i=2}^D \frac{1}{i} \frac{ { \pi_k^N(x)-\pi_{k+1}^N(x) -1     \choose i-2 }{\pi_{k+1}^N(x) \choose D-i } }{ {N  \choose D-2}  } 
\end{eqnarray}

Note that the first summand corresponds to the effective arrival rate to server $1$ associated to the \textcolor{yellow}{yellow} process. The second and the third summands corresponds to the effective arrival rate of the \textcolor{red}{red} process, depending on the queue length of server $N+1$. In the second summand, the expression of the sum corresponds to the probability of choosing $D$ servers in the following manner: choose server $1$ and $N+1$, choose $i-1$ servers with queue length equal to $k$ and choose $(D-1)-i$ servers with queue length strictly bigger than $k$. And then, it is multiplied for the probability of choosing server $1$ between the $i$ servers with length equal to $k$. Analogously, for the third summand, the expression represents the probability of choosing server $1$, $N+1$, other $i-2$ servers with length equal to $k$ and $D-i$ servers with length strictly bigger than $k$. Note that the case $x(N+1) < k$ does not contribute to the rate since an incoming task from the \textcolor{red}{red} process would be always allocated in server $x(N+1)$ instead of $x(1)$ in that situation.\\

Now, that we have a handy way to do a comparison between $\lambda^{N}(x,k)$  and $\lambda^{N+1}(x,k)$, given by \eqref{lambdasum},  \eqref{lambdaN+1} and Lemma \ref{lambdasum}, we are able to do long but elementary computations to prove the next result. The proof is in the Appendix \ref{apconv}.
\begin{lemma}\label{monotone}
There exists a natural number $N_0$ (not depending on $x$ or $k$) such that the sequence $ \{ \lambda^{N}(x,k) \}_{N=N_0}^{\infty} $ is non-decreasing.
\end{lemma} 

Since the sequence  $\{ \lambda^{N}(x,k) \}_{N=1}^{\infty}$ is bounded, we conclude the almost surely convergence for a given state $x$ such that $x(1)=k$. Hence, we have that $\lambda_t^{k,N}$ converges pointwise a.s. in $[0,T]$ to a limit process $\lambda_t^{k}$. \\

It only remains to prove the convergence in the $L^2$ sense and that $\lambda_t^{k}$ is a constant. To demonstrate both assertions it is enough to show that $\left\{\V\left(\lambda_t^{k,N}\right)\right\}_{N \geq 1}$ goes to zero for every $k \in \Z_+$.

In the rest of the proof we restrict ourselves to the case $D=2$ (we present the details for general $D$ in the Appendix \ref{apconv}). By definition, we can compute the variance of the state dependent arrival rate for the tagged queue $X^N(1)$ on the event $\{X^N_t(1)=k\}$ 
\begin{align*}
\V (\lambda_t^{k,N}) & = \V 
\left( \frac{\lambda}{N-1}(\pi^N_k(t) + \pi^N_{k+1}(t) -1) \right) \\
& = \frac{\lambda^2}{(N-1)^2}\left[\V(\pi^N_k) + \V(\pi^N_{k+1}) +2\text{Cov}(\pi^N_k,\pi^N_{k+1})  \right],
\end{align*}
and by Proposition \ref{cotapi}, it follows that
\begin{align*}
\V (\lambda_t^{k,N}) & \leq  \frac{4\lambda^2}{(N-1)^2}\left\{ \dfrac{2(3/2)^DN^2(N-1)}{(N-D)} \left[ N\ln \left(\frac{N+u_N(t)-1}{N}\right) - \dfrac{(N-1)(u_N(t)-1)}{(N+u_N(t)-1)}\right]+ N \right\}\\
& \leq \dfrac{8(3/2)^D\lambda^2N^2}{(N-D)(N-1)} \left[ N\ln \left(\frac{N+u_N(t)-1}{N}\right) - \dfrac{(N-1)(u_N(t)-1)}{(N+u_N(t)-1)}\right] + \dfrac{4\lambda^2N}{(N-1)^2}  ,
\end{align*}
which goes to zero when $N$ goes to infinity. 
\end{proof}

\begin{theorem}\label{theo:conv}
	Let $Q_t$ be a single queue length process with the same distribution service and initial distribution as the tagged queue $X_t^N(1)$ but with (time and state) dependent arrival rate $\lambda_t^k$ (defined in Theorem \ref{convergence}). Then, for any fixed $t \in [0,T]$,
	$$ \lim_{N \rightarrow \infty } || \mathbb P _{ X^N_t(1) }  -  \mathbb P_{Q_t } ||_{TV}  =0 ,$$  
	where $|| \cdot ||_{TV}$ denotes the total variance distance between probability measures. 
\end{theorem}

\begin{proof}
	
We consider the basic coupling between the queue length processes $X_s^N(1)$ and $Q_s$, with the same initial value, service times and the arrival processes with corresponding rates $\lambda_s^{k, N}$ and $\lambda_s^{k}$, for $s \in [0,T]$.\\

The idea of the proof is that if we had uniform convergence for the sequence of functions $ \{ \lambda_s^{Q_s, N} (\omega) \}_{N \geq 0}$ for almost every $\omega$, the total variation convergence will follow easily. Theorem \ref{convergence} does not assure uniform convergence but punctual convergence for almost every $\omega$. However, the monotone convergence in Theorem \ref{convergence}  and a simple application of Egoroff's theorem  \cite[Theorem 2.5.5]{Ash} to the measure space $[0,T]$ equipped with the Lebesgue measure (for a given realization $\omega$), allow us to use similar ideas. \\

 We know that for (almost) every $\omega$, the sequence of functions $ \{ \lambda_s^{Q_s, N} (\omega) \}_{N \geq 0}$ converges to some function $\lambda_s^{Q_s}(\omega)$ defined on $[0,T]$, by Theorem \ref{convergence}.  This implies that, for $\omega$ fixed, we have almost uniform convergence in $[0,T]$ by Egoroff's theorem. Namely, given $\varepsilon >0$ there exist a set $E(\omega) \subseteq [0,T]$ with Lebesgue measure smaller than $\varepsilon$ such that the convergence of $\lambda_s^{Q_s, N} (\omega)$ to $\lambda_s^{Q_s}(\omega)$ is uniform on $[0,T] \setminus E (\omega)$. 
 
 Define $E$ as the random subset of $[0,T]$ such that in the realization $\omega$ takes the value $E(\omega)$.
%Fix $\varepsilon >0$. For (almost) every $\omega$ fixed, the sequence of functions $ \{ \lambda_s^{Q_s, N} (\omega) \}_{N \geq 0}$ converges to some function $\lambda_s^{Q_s}$ defined on $[0,T]$, by Theorem \ref{convergence}.  This implies that, for $\omega$ fixed, we can apply Egoroff's theorem \cite[Theorem 2.5.5]{Ash} to find a set $E(\omega) \subseteq [0,T]$ with Lebesgue measure smaller than $\varepsilon$ such that the convergence of $\lambda_s^{Q_s, N} (\omega)$ to $\lambda_s^{Q_s}$ is uniform on $[0,T] \setminus E (\omega)$. Define $E$ as the random subset of $[0,T]$ such that in the realization $\omega$ takes the value $E(\omega)$.
Then, for any $t \in [0,T]$,
	\begin{eqnarray*}
		|| \mathbb P _{ X^N_t(1) }  -  \mathbb P_{Q_t } ||_{TV}  
		&\leq&  2 \mathbb P ( X^N_t(1) \neq Q_t  )  \\
		&=& 2 \mathbb E(  \mathbb E( 1_{ X_t^N(1)  \neq Q_t } \, | \, \{  \lambda_s^{X_s^N(1), N}, \, \lambda_s^{Q_s}  \, :  s \in [0,t]  \}  ) )\\
			& \leq & 2 \mathbb E ( \mathbb E( 1-   e^{ - \int_{0}^{t}  |  \lambda_s^{Q_s, N} - \lambda_s^{Q_s}  | \, ds  }  | \, \{  \lambda_s^{X_s^N(1), N}, \, \lambda_s^{Q_s}  \, :  s \in [0,t]  \}  ) ) \\
				& = & 2 \mathbb E ( 1-   e^{ - \int_{0}^{t}  |  \lambda_s^{Q_s, N} - \lambda_s^{Q_s}  | \, ds  }   )   \\
				& = & 2 \mathbb E ( 1-   e^{ - 	( \int_{E} |  \lambda_s^{Q_s, N} - \lambda_s^{Q_s}  | \, ds + \int_{[0,t] \setminus E} |  \lambda_s^{Q_s, N} - \lambda_s^{Q_s}  | \, ds ) 	 }   )  \\
		& \leq &     2 \mathbb E (1-   e^{ -( C_D \lambda  \varepsilon +   t \, \sup_{ s \in [0,t] \setminus E} |  \lambda_s^{Q_s, N} - \lambda_s^{Q_s}  | ) }  ) \\
		& \leq & 2 (1-e^{-(  C_D \lambda  \varepsilon + t \varepsilon)} ) \mathbb P \Big( \sup_{ s \in [0,t] \setminus E} | \lambda_s^{Q_s, N} - \lambda_s^{Q_s}    |  < \varepsilon \Big)  \\
		&+&  2 (1- e^{-( C_D \lambda  \varepsilon + t C_D \lambda) }  ) \mathbb P \Big( \sup_{  s \in [0,t] \setminus E } | \lambda_s^{Q_s, N} - \lambda_s^{Q_s}    |  \geq \varepsilon \Big) .
	\end{eqnarray*}  
The second inequality is due to the fact that if all the arrivals coincide for both systems in $[0,t]$ then it must happen that $X^N_t(1) = X^\infty_t$. We used the global bound \eqref{lambdaNbound} for $\lambda_s^{Q_s, N}$ in the last inequalities.

Define 
$$ A_N= \Big\{   \sup_{   s \in [0,t] \setminus E } | \lambda_s^{Q_s, N} - \lambda_s^{Q_s}    |  \geq \varepsilon   \Big\} .$$
By Lemma \ref{monotone}, the functions $\{ \lambda_s^{Q_s, N} \}_{N \geq N_0}$ are non-decreasing, then the sequence of sets $\{ A_N \}_{N \geq N_0}$ is non-increasing. So we can use continuity of the probability and obtain
$$ \lim_{N \rightarrow \infty}  \mathbb P ( \sup_{ s \in [0,t] \setminus E} | \lambda_s^{X_s^N(1), N} - \lambda_s^{Q_s}    |  \geq \varepsilon ) =    \mathbb P \Big(  \bigcap_{N=N_0}^{\infty}  \Big\{ \sup_{  s \in [0,t] \setminus E } | \lambda_s^{X_s^N(1), N} - \lambda_s^{Q_s}    |  \geq \varepsilon \Big\} \Big)  .$$
The uniform convergence of $\{ \lambda_s^{Q_s, N} \}_{N \geq N_0}$ on $E$ implies that last expression is equal to zero, and the result follows.
\end{proof}

We can also state the following properties on the asymptotic
state dependent arrival rate which turns out to be useful in the next Section.

\begin{corollary}\label{limsuplambda} For every fix $t >0$, 
\[	\limsup\limits_{k} \lambda_t^{k} = 0\]
Moreover, if the tagged queue is positive recurrent, then
 \[	\limsup\limits_{k} \lim \limits_{N \to \infty}\lambda_\infty^{k,N} = 0\]

\end{corollary}
\begin{proof}
	By \eqref{lambdaNbound}, we know that 
	\begin{align*}
	 \lambda_t^{k,N} \leq   \lambda  &  \Big( \frac{N}{N-D+1}\Big) \Big( \frac{N}{N-D+2}\Big) \cdots
	 \Big( \frac{N}{N-1}\Big) \sum_{i=0}^{D-1}  \Big( \frac{ \pi_{k+1}^N(t)}{N} \Big)  \Big( \frac{ \pi_{k+1}^N(t) -1}{N} \Big) \cdots
	 \Big( \frac{ \pi_{k+1}^N(t)- (i-1)}{N} \Big) \\
	  \cdot &   \Big( \frac{ \pi_k^N(t)- (i+1)}{N} \Big)  \Big( \frac{ \pi_k^N(t)- (i+2)}{N} \Big) \cdots  \Big( \frac{ \pi_k^N(t)- (D-1)}{N} \Big).
	\end{align*}
	As $\lim\limits_{N \rightarrow \infty} \V \left(\pi_{k}^N(t)/N\right)=0$ by Proposition \ref{cotapi},
	 it is enough to show that 
	\[\limsup\limits_{k}\lim\limits_{N \rightarrow \infty} \E\left(\dfrac{\pi_{k}^N(t)}{N}\right) = 0.\] 
	For every fix $k \in \NN$, under the assumption of exchangeable initial distribution $X^N_0$, we have that 
	\[\lim\limits_{N \rightarrow \infty} \E\left(\dfrac{\pi_{k}^N(t)}{N}\right)=\lim\limits_{N \rightarrow \infty} \dfrac{\sum\limits_{i=1}^N\p\left(X^N_t(i) \geq k \right)}{N}= \p\left(Q_t \geq k\right),	\]
		where $Q_t$ is the process defined in Theorem \ref{theo:conv}.
		 As for fixed $t$, the process $Q$ can be bounded by a Poisson process with fixed intensity using bound (3.3), we have that  $\limsup\limits_{k}\p\left(Q_t \geq k\right)=0$ and the desired result follows.
		 For the stationary version, this a direct consequence of the assumption of stability of the tagged queue.	
\end{proof}

\section{Convergence in the stationary regime}\label{sec:stat}

The previous results imply convergence of the stationary version of a finite subset of $l$ queues to a product measure $\rho^{\bigotimes l}$ where $\rho$ is the stationary measure of the tagged particle.
One should however notice that this is so because the correlation estimates {\bf do not depend} on the initial conditions of the system. 

Indeed, by \cite{B}, there exists a stationary measure for the system of $N$-queue while the existence of a stationary measure of the limiting tagged queue is easily obtained via a Lyapunov criterion.
We can hence state the following Proposition.

\begin{proposition}
Under the stationary regime, for any fixed $k$, $ \Big(X^N_{\infty}(1), \ldots X^N_{\infty}(l) \Big)$ converges in total variation to $\rho^{\bigotimes l}$,
where $\rho$ is the stationary distribution of the tagged queue.
Moreover if the service policy is symmetric, $\rho([k,\infty))=\lambda^{(D^k-1)/(D-1)}$.
Finally if the service time distribution has an exponential moment, then there exist positive constants $c_1,c_2$ and $\alpha>0$ such that:
$$|| \mathbb P _{ X^N_{\infty} (1) }   -  \rho ||_{TV} \le  c_1 \ln (N)/N + c_2 N^{-\alpha}.$$ 
\end{proposition}

\begin{remark}
The constant $\alpha$ is determined by the speed of convergence to stationarity of the cavity process.
\end{remark}

\begin{proof}
Starting the system in stationary regime, that is $X^N_t (1)\overset{\mathcal{D}}{=} X^N_{\infty} (1)$, and using the previous results (which can be done using the uniformity on the initial distribution): 
\begin{equation}\label{triangle}
     || \mathbb P^{\pi^N} _{ X^N_{t} (1) }   -  \rho ||_{TV} \leq   || \mathbb P _{ X^N_{t} (1) }  -  \mathbb P_{X^{\infty}_t  } ||_{TV}   +  || \mathbb P _{ X^{\infty}_{t}  }   -  \rho ||_{TV}  .
\end{equation}

The first term can be controlled as previously by our correlation estimates. Hence the result follows if we can control the second one, that is, to prove positive recurrence of the tagged queue. This is done in the Appendix \ref{appstat}. In the case of symmetric service policy, the results of \cite{Z} allow to conclude insensitivity of the tagged queue and hence aysmptotic insensitivity of the system. If in addition, the service time distribution has an exponential moment, one can show (see \ref{appstat}) that the convergence towards the stationary measure is exponential, and hence choosing $t=T_N= c \log(N)$ in the previous equality gives the desired result.

% Now, since Theorem \ref{convergence} gives us a rate $\lambda_t^\infty$  independent of the distribution of the services we can put exponential services. In that case, \cite{VDK} proved the convergence of the third term of the r.h.s. in \eqref{triangle}, so there exists $T_2$ such that for all $t \geq T_2$ we have that
% $$ || \mathbb P _{ X^{\infty}_{t} (1) }   -  \rho ||_{TV}  < \frac{\varepsilon}{3}.$$
% Now, take $T= \max \{T_1,T_2\}$. By the first part, we can choose $N$ such that for all $M \geq N$ we have 
% $$ || \mathbb P _{ X^M_{T} (1) }  -  \mathbb P_{X^{\infty}_T  } ||_{TV} < \frac{\varepsilon}{3}, $$
% and we are done.

\end{proof}

% In the particular case of a PS-queue with exponential services the candidate \eqref{mk} is indeed the stationary measure.
% 	
% The intuition of last remark is the following. Assume that almost surely
% \begin{equation}\label{assu}
% \lim_{N \rightarrow \infty}	\frac{ \pi_k(X_t^N)   }{N}  = C \phi_k ,
% \end{equation}
% for some positive constants $\phi_k$ and some normalizing constant $C$ (which depends on $D$). 
% 
% The generalized balance equations \cite{Z} for the tagged queue $X^N(1)$ with exponential services of rate $1$ take the form
% $$   \mu^N( k+1)  \cdot 1 =  \mu^N (k) \cdot   \lambda_t^{X^N(1)}   \qquad  \forall k \geq 0 . $$
% Taking the limit when $N$ grows to infinity, under the assumption \eqref{assu}, one obtains the balance equations for the idealized infinite system:
% $$  \mu( k+1)  =  \mu (k) \cdot  \lambda D  [ C \phi_k  ]^{D-1}  \qquad  \forall k \geq 0  .$$
% By chaos propagation, we have that $\lim_{N \rightarrow \infty } m (X_t) (k) =  \mu(k) $, therefore the solution for last equations must also satisfy that
% $$ \mu(k) = C_2 ( \phi_{k+1} - \phi_k).$$ 
% By choosing $\phi_k= C \lambda^{\frac{D^k-1}{D-1} } $ and the distribution $\mu(k)= C_3 (\lambda^{\frac{D^k-1}{D-1} }  -  \lambda^{\frac{D^{k+1}-1}{D-1} } ) $ both conditions are met.
% 

\appendix
\section{Appendix}

\subsection{Propagation of Chaos}\label{approp}

\subsubsection{Proof of Lemma \ref{cota} for general $D$:}

	First we show how to obtain \eqref{clanbound}. Note that	
	\begin{align*}
	\dfrac{d\mathbb{E}\left[|\psi^i(t)|\big|\sigma(\omega[-t,0))\right]}{dt}& =\lambda N\sum\limits_{k=1}^D (D-k)\mathbb{P}[\text{choose k servers in } \psi^i(t) ] 
	=\lambda N\sum\limits_{k=1}^{D-1} (D-k)\left(\dfrac{{|\psi^i(t)| \choose k}{N- |\psi^i(t)| \choose D-k}}{{N \choose D}}\right)\\
	& = \lambda N\sum\limits_{r=1}^{D-1} r\left(\dfrac{{|\psi^i(t)| \choose D-r}{N- |\psi^i(t)| \choose r}}{{N \choose D}}\right)\\
	& = \lambda N\left( \dfrac{(N- |\psi^i(t)| )D}{N} - D\mathbb{P}[\text{don't choose any servers in } \psi^i(t) ]\right)\\
	& = \lambda DN \left( 1  - \dfrac{|\psi^i(t)|}{N} - \mathbb{P}[\text{don't choose any servers in } \psi^i(t) ]  \right) .
	%&\leq \lambda D \left[N \left(1 - \dfrac{\mathbb{E}\left[{N-|\psi^i(t)|\choose D}\right]}{{N \choose D}} \right) - \mathbb{E}[|\psi^i(t)|] \right]
	\end{align*}
	Observe that if $|\psi^i(t)|> N - D$,  $\mathbb{P}[\text{don't choose any servers in } \psi^i(t) ]= 0$ while otherwise:
	\begin{align*}
	\mathbb{P}[\text{don't choose any servers in } \psi^i(t) ] & = \dfrac{{N-|\psi^i(t)| \choose D }}{{N \choose D}}
	= \dfrac{(N-|\psi^i(t)|)!}{(N-|\psi^i(t)|-D)!}\dfrac{(N-D)!}{N!}\\
	& = \dfrac{N-|\psi^i(t)|}{N} \dfrac{N-|\psi^i(t)|-1}{N-1}\cdots \dfrac{N-|\psi^i(t)|-D+1}{N-D+1}\\
	&=  \left(1 - \frac{|\psi^i(t)|}{N} \right)\left(1 - \frac{|\psi^i(t)|}{N-1}\right)\cdots\left(1-\frac{|\psi^i(t)|}{N-D+1}\right)\\
	&=  1 - | \psi^i(t)| \Big( \sum_{k=0}^{D-1} \frac{1}{N-k}  \Big) + | \psi^i(t)|^2 \Big( \sum_{ k,h \in \{ 0,...,D-1\} , \, k \neq h} \frac{1}{ (N-k)(N-h) }  \Big) \\
	&+ \dots   + (-1)^D  | \psi^i(t)|^{D}  \frac{1}{N (N-1) \cdots (N-D+1)},
	\end{align*}
	where the last equality is proven by induction. Hence, we obtain the following bound for the derivative	
	\begin{align*}
	\dfrac{d\mathbb{E}\left[|\psi^i(t)|\big|\sigma(\omega[-t,0))\right]}{dt}  
	& = \lambda DN \left( 1  - \dfrac{|\psi^i(t)|}{N} - \mathbb{P}[\text{don't choose any servers in } \psi^i(t) ]  \right) \\
	& = \lambda DN \left(    1  - \dfrac{|\psi^i(t)|}{N} -  \textbf{1}_{|\psi^i(t)|\leq N-D}\left(1 - \frac{|\psi^i(t)|}{N} \right) \left(1 - \frac{|\psi^i(t)|}{N-1}\right)  \cdots 
	\left(1 - \frac{|\psi^i(t)|}{N-D + 1} \right) \right) \\
	& = \lambda DN  \left( 1  - \dfrac{|\psi^i(t)|}{N} \right) \left[1   -   \textbf{1}_{|\psi^i(t)|\leq N-D}\left(1 - \frac{|\psi^i(t)|}{N-1}\right)  \cdots \left(1 - \frac{|\psi^i(t)|}{N-D + 1} \right) \right] 	\\	
	& \leq \lambda DN  \left( 1  - \dfrac{|\psi^i(t)|}{N} \right)\left[\textbf{1}_{|\psi^i(t)|> N-D} + \textbf{1}_{|\psi^i(t)|\leq N-D}\sum\limits_{k=1}^{D-1}{D-1 \choose k}\left(\dfrac{|\psi^i(t)|}{N-D}\right)^k\right]	\\  
	&\leq \dfrac{2^{D-1}\lambda DN}{N-D}|\psi^i(t)|\left( 1  - \dfrac{|\psi^i(t)|}{N} \right).
	\end{align*}
	In the last inequality we used that $\dfrac{1}{N-k}\leq \dfrac{1}{N-D}$ for all $k \in \{1,2,\cdots,D-1\}$ and the Newton's general binomial theorem:
	\begin{align*}  
	\dfrac{N-D}{|\psi^i(t)|}\textbf{1}_{|\psi^i(t)|> N-D} + \textbf{1}_{|\psi^i(t)|\leq N-D}\sum\limits_{k=1}^{D-1}{D-1 \choose k}\left(\dfrac{|\psi^i(t)|}{N-D}\right)^{k-1}
	&\leq \textbf{1}_{|\psi^i(t)|> N-D} + \textbf{1}_{|\psi^i(t)|\leq N-D}\sum\limits_{k=1}^{D-1}{D-1 \choose k}\\
	&\leq  \sum\limits_{k=0}^{D-1}{D-1 \choose k} = 2^{D-1}.
	\end{align*}
 Therefore \eqref{clanbound} follows. \\ 

Now we prove \eqref{cota2}. We have
	 
	\begin{align*}
& \dfrac{d\mathbb{P}(\psi^i(s) \cap \psi^j(s) \neq \emptyset \big| \sigma(\omega[-s, 0)))}{ds} = \lambda N \mathbb{P}\left[\psi^i(s) \cap \psi^j(s) \neq \emptyset \big|\psi^i(s^-) \cap \psi^j(s^-) = \emptyset \right]\\
	=& \lambda N  \left[1- \dfrac{{N -|\psi^i(s^-)| -|\psi^j(s^-)| \choose D }}{{N \choose D}} - \sum\limits_{r=1}^D\dfrac{{N -|\psi^i(s^-)| -|\psi^j(s^-)| \choose D -r }{|\psi^i(s^-)| \choose r}}{{N \choose D}} -  \sum\limits_{r=1}^D\dfrac{{N -|\psi^i(s^-)| -|\psi^j(s^-)| \choose D -r }{|\psi^j(s^-)| \choose r}}{{N \choose D}}\right] \\
	\cdot& \textbf{1}_{\psi^i(s^-)\cap \psi^j(s^-)=\emptyset} \\
	=& \lambda N  \left[1+ \dfrac{{N -|\psi^i(s^-)| -|\psi^j(s^-)| \choose D }}{{N \choose D}} -\dfrac{{N-|\psi^j(s^-)| \choose D}}{{N \choose D}} - \dfrac{{N-|\psi^i(s^-)| \choose D}}{{N \choose D}}  \right]\textbf{1}_{\psi^i(s^-)\cap \psi^j(s^-)=\emptyset},
	\end{align*}	
	where we used that
	\begin{align*}
	\sum\limits_{r=1}^D\dfrac{{N -|\psi^i(s^-)| -|\psi^j(s^-)| \choose D -r }{|\psi^j(s^-)| \choose r}}{{N \choose D}} &  = \dfrac{{N-|\psi^i(s^-)| \choose D}}{{N \choose D}}\sum\limits_{r=1}^D\dfrac{{N -|\psi^i(s^-)| -|\psi^j(s)| \choose D -r }{|\psi^j(s^-)| \choose r}}{{N-|\psi^i(s^-)| \choose D}} \\
	&= \dfrac{{N-|\psi^i(s^-)| \choose D}}{{N \choose D}} \left[1- \dfrac{{N -|\psi^i(s)| -|\psi^j(s^-)| \choose D}}{{N-|\psi^i(s^-)| \choose D}} \right]\\
	&= \dfrac{{N-|\psi^i(s^-)| \choose D}}{{N \choose D}} - \dfrac{{N -|\psi^i(s)| -|\psi^j(s^-)| \choose D}}{{N \choose D}}.
	\end{align*}		
	
	Hence, using the same strategy to prove \eqref{clanbound} above,  
	\begin{align*}
&\dfrac{d\mathbb{P}(\psi^i(s) \cap \psi^j(s) \neq \emptyset \big| \sigma(\omega[-s, 0)))}{ds} = 
\lambda N  \left[1+ \dfrac{{N -|\psi^i(s^-)| -|\psi^j(s^-)| \choose D }}{{N \choose D}} -\dfrac{{N-|\psi^j(s^-)| \choose D}}{{N \choose D}} - \dfrac{{N-|\psi^i(s^-)| \choose D}}{{N \choose D}} \right] \\
&\cdot \textbf{1}_{\psi^i(s^-)\cap \psi^j(s^-)=\emptyset} \\
	=  &\lambda N  \left[1 +  \textbf{1}_{\{|\psi^i(s^-)|+|\psi^j(s^-)|\leq N-D\}}\left(1-\sum\limits_{k=0}^{D-1}\frac{|\psi^i(s^-)|+|\psi^j(s^-)|}{N-k} +\cdots + (-1)^D\frac{(|\psi^i(s^-)|+|\psi^j(s^-)|)^D}{N(N-1)\cdots(N-D+1)}\right)\right.\\
	& - \textbf{1}_{\{|\psi^i(s^-)|\leq N-D\}} \left(1-\sum\limits_{k=0}^{D-1}\frac{|\psi^i(s^-)|}{N-k} +\cdots + (-1)^D\frac{|\psi^i(s^-)|^D}{N(N-1)\cdots(N-D+1)}\right) \\
	& \left. -\textbf{1}_{\{|\psi^j(s^-)|\leq N-D\}}\left(1-\sum\limits_{k=0}^{D-1}\frac{|\psi^j(s^-)|}{N-k} +\cdots + (-1)^D\frac{|\psi^j(s^-)|^D}{N(N-1)\cdots(N-D+1)}\right)\right] \textbf{1}_{\{\psi^i(s^-)\cap \psi^j(s^-)=\emptyset\}}.  		
	\end{align*}
	Since $ \textbf{1}_{\{|\psi^i(s^-)|+|\psi^j(s^-)|\leq N-D\}} \leq  \textbf{1}_{\{|\psi^i(s^-)|\leq N-D\}}$ and $ \textbf{1}_{\{|\psi^i(s^-)|+|\psi^j(s^-)|\leq N-D\}} \leq  \textbf{1}_{\{|\psi^i(s^-)|\leq N-D\}}$,
	it follows that
	\begin{align*}
&\dfrac{d\mathbb{P}(\psi^i(s) \cap \psi^j(s) \neq \emptyset \big| \sigma(\omega[-s, 0)))}{ds}  \leq  \lambda N  \left[1 + \textbf{1}_{\{|\psi^i(s^-)|+|\psi^j(s^-)|\leq N-D\}} \left(1-\sum\limits_{k=0}^{D-1}\frac{|\psi^i(s^-)|+|\psi^j(s^-)|}{N-k} +\cdots \nonumber\right.\right.\\
	&+ (-1)^D\frac{(|\psi^i(s^-)|+|\psi^j(s^-)|)^D}{N(N-1)\cdots(N-D+1)} - 1+ \sum\limits_{k=0}^{D-1}\frac{|\psi^j(s^-)|}{N-k} -\cdots - (-1)^D\frac{|\psi^j(s^-)|^D}{N(N-1)\cdots(N-D+1)} \\
	& \left.\left. -1+\sum\limits_{k=0}^{D-1}\frac{|\psi^i(s^-)|}{N-k} -\cdots - (-1)^D\frac{|\psi^i(s^-)|^D}{N(N-1)\cdots(N-D+1)}\right)\right]\textbf{1}_{\{\psi^i(s^-)\cap \psi^j(s^-)=\emptyset\}}\\
	\leq  &\lambda N  \left[\textbf{1}_{\{|\psi^i(s^-)|+|\psi^j(s^-)|> N-D\}} + \textbf{1}_{\{|\psi^i(s^-)|+|\psi^j(s^-)|\leq N-D\}}\left(\sum\limits_{k=0}^{D-1}\sum\limits_{h=0,h\neq k}^{D-1}\frac{2|\psi^i(s^-)||\psi^j(s^-)|}{(N-k)(N-h)} +\cdots \nonumber \right. \right.\\
	& \left.\left. + (-1)^D\frac{(|\psi^i(s^-)|+|\psi^j(s^-)|)^D - |\psi^j(s)|^D - |\psi^i(s^-)|^D }{N(N-1)\cdots(N-D+1)}\right)\right]\textbf{1}_{\{\psi^i(s^-)\cap \psi^j(s^-)=\emptyset\}} \\
	\leq &  \lambda N  \left[\textbf{1}_{\{|\psi^i(s^-)|+|\psi^j(s^-)|> N-D\}} + \textbf{1}_{\{|\psi^i(s^-)|+|\psi^j(s^-)|\leq N-D\}}\sum\limits_{k=2}^D{D \choose k}\sum\limits_{r=1}^{k-1}{k\choose r} \frac{|\psi^i(s^-)|^r|\psi^j(s^-)|^{k-r}}{(N-D)^k}\right] \\
	&\cdot \textbf{1}_{\{\psi^i(s^-)\cap \psi^j(s^-)=\emptyset\}}.
	\end{align*}
	
	Now we can bound
	\begin{align*}
	\lambda N & \left[\textbf{1}_{\{|\psi^i(s^-)|+|\psi^j(s^-)|> N-D\}} + \textbf{1}_{\{|\psi^i(s^-)|+|\psi^j(s^-)|\leq N-D\}} \sum\limits_{k=2}^D{D \choose k}\sum\limits_{r=1}^{k-1}{k\choose r} \frac{|\psi^i(s^-)|^r|\psi^j(s^-)|^{k-r}}{(N-D)^k}\right]\\
	= &\dfrac{\lambda N|\psi^i(s^-)||\psi^j(s^-)| }{(N-D)^2}\left[\textbf{1}_{\{|\psi^i(s^-)|+|\psi^j(s^-)|> N-D\}} \frac{(N-D)^2}{|\psi^i(s^-)||\psi^j(s^-)|}\right. \\
	&\left. + \textbf{1}_{\{|\psi^i(s^-)|+|\psi^j(s^-)|\leq N-D\}}\sum\limits_{k=2}^D{D \choose k} \sum\limits_{r=1}^{k-1} {k\choose r}\frac{|\psi^i(s^-)|^{r-1}|\psi^j(s^-)|^{k-r-1}}{(N-D)^{k-2}}\right]\\
	\leq & \dfrac{\lambda N|\psi^i(s^-)||\psi^j(s^-)| }{(N-D)^2}\left[\textbf{1}_{\{|\psi^i(s^-)|+|\psi^j(s^-)|> N-D\}}  + \textbf{1}_{\{|\psi^i(s^-)|+|\psi^j(s^-)|\leq N-D\}}\sum\limits_{k=2}^D{D \choose k}\sum\limits_{r=1}^{k-1}{k\choose r}\right] \\
	\leq & \dfrac{3^D\lambda N|\psi^i(s^-)||\psi^j(s^-)| }{(N-D)^2}.
	\end{align*}
	From this bound, \eqref{cota2} follows in a similar way as in the case $D=2$. 
\begin{flushright}
$ \square$
\end{flushright}

\subsection{Convergence} \label{apconv}

\subsubsection*{Proof of Lemma \ref{lambdasum}:}

Note that:
\begin{align}\label{lambdahy}
\lambda^N (x,k)& =  \lambda N\sum\limits_{i=1}^D \frac{1}{i}  \frac{ { \pi_k^N(x)-\pi_{k+1}^N(x) -1    \choose i-1 }{\pi_{k+1}^N(x)\choose D-i }}{ {N  \choose D}}
= \dfrac{\lambda N}{{N  \choose D}} \, \frac{{ \pi_k^N(x) \choose D}}{(\pi_k^N(x)-\pi_{k+1}^N(x))} \,  \sum\limits_{i=1}^D \frac{ { \pi_k^N(x)-\pi_{k+1}^N(x)   \choose i }{\pi_{k+1}^N(x)\choose D-i }}{ { \pi_k^N(x) \choose D}} \nonumber   \\
& =  \dfrac{\lambda N}{{N  \choose D}} \frac{{ \pi_k^N(x) \choose D}}{(\pi_k^N(x)-\pi_{k+1}^N(x))}\left[1- \frac{ {\pi_{k+1}^N(x)\choose D}}{ { \pi_k^N(x) \choose D}}\right]
= \dfrac{\lambda N}{ {N  \choose D}  }   \left[  \frac{  { \pi_k^N(x) \choose D} -  {\pi_{k+1}^N(x)\choose D}  }{ \pi_k^N(x)-\pi_{k+1}^N(x)  }  \right].
\end{align}
On the other hand, for $D \leq a<b$ and putting $c=b-a$, we have that
$$ a( a+ c -1 ) ... (a+c -(D-1)) = a(a-1)...(a-(D-1)) + c S^D(a,b) $$
by just expanding conveniently the product of the l.h.s.. Put $a=\pi_{k+1}^N(x)$ and $b=\pi_{k}^N(x)$. Then last formula translates in
$$ D! {\pi_{k}^N(x) \choose D } = D! {\pi_{k+1}^N(x)\choose D }  +  (\pi_k^N(x)-\pi_{k+1}^N(x) )  S(\pi_{k+1}^N(x), \pi_{k}^N(x)),$$
and plugging this in \eqref{lambdahy} gives the desired conclusion.
\begin{flushright}
	$\square$	
\end{flushright}

\subsubsection*{Proof of Lemma \ref{monotone}:}

Define 
$$ B_1 :=  \I_{ \{ x(1)=k, x(N+1) > k \} }  (\lambda D )  \sum \limits_{i=1}^D \frac{1}{i} \frac{ { \pi_k^N(x)-\pi_{k+1}^N(x) -1    \choose i-1 }{\pi_{k+1}^N(x) \choose D-1-i }}{ {N  \choose D-1}  } $$ 
and
$$ B_2 :=  \I_{ \{ x(1)=k, x(N+1) = k  \} } (\lambda D)  \sum \limits_{i=2}^D \frac{1}{i} \frac{ { \pi_k^N(x)-\pi_{k+1}^N(x) -1    \choose i-2 }{\pi_{k+1}^N(x) \choose D-i } }{ {N  \choose D-2}  }  $$ 
which correspond to the the second and third summand of the r.h.s. of equation \eqref{lambdaN+1}.

We first compute the value of $B_1$ and $B_2$. Note that $B_1$ has a similar expression as the one for $\lambda^{N}(x,k)$ in Lemma \ref{lambdasum}: we have rate $\lambda D$ instead of $\lambda N$, the value $D-1$ instead of $D$ and an extra indicator function $\I _{ \{ x(N+1) > k  \}}$. Hence
$$ B_1 =  \I_{ \{ x(1)=k, x(N+1) > k  \} } (\lambda D) \frac{ (N - (D-1))! }{N !}  S^{D-1} ( \pi^{N}_{k+1} (x), \pi^{N}_{k} (x) ) .$$

On the other hand, we can obtain a lower bound for $B_2$:
\begin{eqnarray*}
 B_2 & := &  \I_{ \{ x(1)=k, x(N+1) = k  \}} (\lambda D)  \sum \limits_{i=2}^D \frac{1}{i} \frac{ { \pi_k^N(x)-\pi_{k+1}^N(x) -1    \choose i-2 }{\pi_{k+1}^N(x) \choose D-i } }{ {N  \choose D-2}  }  \\
 &=& \I_{ \{ x(1)=k, x(N+1) = k  \}}  \frac{ \lambda D}{ { N \choose D-2 }}  \sum \limits_{j=1}^{D-1} \Big( \frac{j}{j+1} \Big) \frac{1}{j}  { \pi_k^N(x)-\pi_{k+1}^N(x) -1    \choose j-1 }  {\pi_{k+1}^N(x) \choose D-1-j }.
\end{eqnarray*}
Since $ \frac{j}{j+1} \geq \frac{1}{2}$ for all $j=1,...,D-1$, we have that
\begin{eqnarray}\label{boundB_2}
	B_2 & \geq  & \I_{ \{ x(1)=k, x(N+1) = k  \} }  \frac{ \lambda D}{ 2 { N \choose D-2 }} 
	 \sum \limits_{j=1}^{D-1}  \frac{1}{j}  { \pi_k^N(x)-\pi_{k+1}^N(x) -1  \choose j-1 } {  {\pi_{k+1}^N(x) \choose D-1-j } }  \nonumber \\
 & = & \I_{ \{ x(1)=k, x(N+1) = k  \}}  \frac{ \lambda D}{ 2 { N \choose D-2 }} \frac{ {\pi_{k}^N(x) \choose D-1 }  }{ \pi_k^N(x)-\pi_{k+1}^N(x) } 
 \Big(  1 - \frac{ {\pi_{k+1}^N(x) \choose D-1 }    }{ {\pi_{k}^N(x) \choose D-1 } } \Big)  \nonumber \\
 & =&  \I_{ \{x(1)=k, x(N+1) = k  \} }  \frac{ \lambda D}{ 2 { N \choose D-2 }}
 \Big(   \frac{ {\pi_{k}^N(x) \choose D-1 } - {\pi_{k+1}^N(x) \choose D-1 }  }{   \pi_k^N(x)-\pi_{k+1}^N(x)  } \Big)  \nonumber \\
 &= &  \I_{ \{ x(1)=k, x(N+1) = k  \}}   \frac{ \lambda D}{ 2 { N \choose D-2 }}  \frac{ S^{D-1}( \pi_{k+1}^N(x), \pi_{k}^N(x)  )  }{ (D-1)!} \nonumber \\
 &=&  \I_{ \{x(1)=k, x(N+1) = k  \}}  \frac{ \lambda D}{ 2 (D-1)} \frac{ (N-(D-2))!}{N!}  S^{D-1}( \pi_{k+1}^N(x), \pi_{k}^N(x)  ),
 \end{eqnarray}
where we used Lemma \ref{lambdasum} at the end. 

Now we can compare $\lambda^{N+1}(x,k)$ and $\lambda^{N}(x,k)$. We do it in cases, depending on the value of $x(N+1)$. Note that the effective arrival rate to the server $1$ associated to the red process is given by
$$  \I_{ \{x(1)=k \} } \lambda (D-1)  \frac{ (N-D)! }{N!} S^D ( \pi_{k+1}^N (x),\pi_{k}^N(x) ).$$ 

In the case $x(N+1)>k$, we have 
\begin{eqnarray*}
\lambda^{N+1}(x,k)- \lambda^{N}(x,k) & = &  B_1 - \lambda (D-1)  \frac{ (N-D)! }{N!} S^D ( \pi_{k+1}^N (x),\pi_{k}^N(x) )      \\
& = &    \lambda D \frac{(N-(D-1))!}{N!} S^{D-1}( \pi_{k+1}^N(x), \pi_k ^N(x)) - \lambda (D-1)  \frac{ (N-D)! }{N!} S^D ( \pi_{k+1}^N (x),\pi_{k}^N(x) )   \\
& = & \lambda  \frac{ (N-D)! }{N!}  \Big( D(N-(D-1)) S^{D-1}  ( \pi_{k+1}^N(x), \pi_k ^N(x)) -(D-1) S^{D} ( \pi_{k+1}^N(x), \pi_k ^N(x)) \Big).
\end{eqnarray*}
Using the identity
\begin{equation}\label{SD}
S^D ( \pi_{k+1}^N(x), \pi_k ^N(x)) = (  \pi_k ^N(x) -(D-1) ) S^{D-1} ( \pi_{k+1}^N(x), \pi_k ^N(x)) +  \pi_{k+1}^N(x) ( \pi_{k+1}^N(x) -1 )  \dots ( \pi_{k+1}^N(x)-(D-2)) ,
\end{equation}
it follows that 
\begin{eqnarray}\label{SD-SD-1}
	\lambda^{N+1}(x,k)- \lambda^{N}(x,k) & =  & \lambda  \frac{ (N-D)! }{N!} \Big[ (D(N-(D-1))-(D-1) \Big( (  \pi_k ^N(x) - (D-1) )   S^{D-1}  ( \pi_{k+1}^N(x), \pi_k ^N(x)) \nonumber \\
	&- &  \pi_{k+1}^N(x) ( \pi_{k+1}^N(x) -1 )  \dots ( \pi_{k+1}^N(x)-(D-2)) \Big)  \Big]. 
\end{eqnarray}
On the other hand, since $\pi_{k}^N(x)-\pi_{k+1}^N(x) \geq 1$, we can bound 
$$ \pi_{k+1}^N(x) ( \pi_{k+1}^N(x) -1 )  \dots ( \pi_{k+1}^N(x)-(D-2)) $$
$$ \leq \Big( \pi_{k+1}^N(x) ( \pi_{k+1}^N(x) -1 )  \cdots   ( \pi_{k+1}^N(x) -(D-3) ) \Big)   ( \pi_{k}^N(x) -(D-1) ) .$$
Note that $\pi_{k+1}^N(x) ( \pi_{k+1}^N(x) -1 )  \dots   ( \pi_{k+1}^N(x) -(D-3) )$ is the last term of the sum $S^{D-1}  ( \pi_{k+1}^N(x), \pi_k ^N(x))$ and its value is less or equal than the value of any of the other summands of $S^{D-1}  ( \pi_{k+1}^N(x), \pi_k ^N(x))$. The sum $S^{D-1}  ( \pi_{k+1}^N(x), \pi_k ^N(x))$ has $D-1$ summands, then
\begin{equation}\label{boundprod}
\pi_{k+1}^N(x) ( \pi_{k+1}^N(x) -1 )  \dots ( \pi_{k+1}^N(x)-(D-2)) \leq \frac{ ( \pi_{k}^N(x) -(D-1)) S^{D-1}  ( \pi_{k+1}^N(x), \pi_k ^N(x)) }{D-1}
\end{equation}
Plugging last inequality in \eqref{SD-SD-1} we obtain that
\begin{eqnarray*}
\lambda^{N+1}(x,k)- \lambda^{N}(x,k) & \geq  & \lambda  \frac{ (N-D)! }{N!} \Big[ D(N-(D-1)) \\
&-& (D-1) (  \pi_k ^N(x) - (D-1) ) (1 +  \frac{1}{D-1} ) S^{D-1}  ( \pi_{k+1}^N(x), \pi_k ^N(x)) \Big]  \\
& = &  \lambda  \frac{ (N-D)! D  S^{D-1}  ( \pi_{k+1}^N(x), \pi_k ^N(x)) }{N!} (N - \pi_k ^N(x) ) \geq 0,
\end{eqnarray*}
for all $N \geq 0$. \\

In the case $x(N+1)=k$, we have 
\begin{eqnarray*}
	\lambda^{N+1}(x,k)- \lambda^{N}(x,k) & = & B_2 - \lambda (D-1)  \frac{ (N-D)! }{N!} S^D ( \pi_{k+1}^N (x),\pi_{k}^N(x) )  \\
	& \geq &  \frac{ \lambda (N-D)! }{N!}  \Big(  (N-(D-2))(N-(D-1)) \frac{D}{2(D-1)} S^{D-1} ( \pi_{k+1}^N (x),\pi_{k}^N(x) ) \\
	&-& (D-1) S^D ( \pi_{k+1}^N (x),\pi_{k}^N(x) ) \Big),
\end{eqnarray*}
where we used \eqref{boundB_2}. Similar to the other case, we use \eqref{SD} and \eqref{boundprod} to obtain 
\begin{eqnarray*}
	\lambda^{N+1}(x,k)- \lambda^{N}(x,k) & \geq  & \frac{ \lambda (N-D)! }{N!} \Big(  (N-(D-2))(N-(D-1)) \frac{D}{2(D-1)} S^{D-1} ( \pi_{k+1}^N (x),\pi_{k}^N(x) ) \\
	&-& (D-1) ( \pi_{k}^N(x) -(D-1) ) (1+ \frac{1}{D-1} ) S^{D-1} ( \pi_{k+1}^N (x),\pi_{k}^N(x) ) \Big)  \\
	& \geq  & \frac{ \lambda (N-D)! }{N!} \Big(  (N-(D-2))(N-(D-1)) \frac{D}{2(D-1)} S^{D-1} ( \pi_{k+1}^N (x),\pi_{k}^N(x) ) \\
	&-& (D-1) ( N-(D-1) ) ( \frac{D}{D-1} ) S^{D-1} ( \pi_{k+1}^N (x),\pi_{k}^N(x) ) \Big)  \\
	&=& \frac{ \lambda (N-(D-1))! D S^{D-1} ( \pi_{k+1}^N (x),\pi_{k}^N(x) ) }{2 (D-1)N!} (N+4-3D), 
\end{eqnarray*}
hence by putting $N_0=3D-4$ the proof is complete.
\begin{flushright}
	$\square$	
\end{flushright}

\subsubsection*{Proof of Theorem \ref{convergence} for general $D$:}

Using Lemma \ref{lambdasum} we have that
$$ \lambda_t^{k,N}=  1_{ \{X_t^N(1)=k \} }\lambda N \dfrac{(N-D)!}{N!}S(\pi^N_{k+1}(t),\pi^N_{k}(t) ), $$
where the function $S$ is defined by \eqref{eqsum}. Then 
\[ \V(\lambda_t^{k,N})=\lambda^2 \left(\dfrac{(N-D)!}{(N-1!)}\right)^2 \V (S(\pi^N_{k+1}(t),\pi^N_{k}(t) )).\] 
Using the bilinearity of covariance, we can rewrite the r.h.s. of the last equation as sum of  $C_1 =C_1(D)$ terms of the form 
$$\Cov\left((\pi^N_{k}(t))^n(\pi^N_{k+1}(t))^m,(\pi^N_{k}(t))^r(\pi^N_{k}(t))^s\right),$$
 where $0\leq n+m,r+s\leq D$. By definition,
\begin{align*}
\Cov&\left((\pi^N_{k}(t))^n(\pi^N_{k+1}(t))^m,(\pi^N_{k}(t))^r(\pi^N_{k}(t))^s\right) \\
&= \mathbb{E}\left[(\pi^N_{k}(t))^{n+r}(\pi^N_{k+1}(t))^{m+s}\right] - 
\mathbb{E}\left[(\pi^N_{k}(t))^n(\pi^N_{k+1}(t))^m\right]
\mathbb{E}\left[(\pi^N_{k}(t))^r(\pi^N_{k+1}(t))^s\right],
\end{align*}
and by the definition of $\pi^N_{k}(t)$, this is less or equal to $N^D C_2(D)$ terms of the form 
\begin{equation*}
\begin{aligned}
\mathbb{P}&\left[X_t(i_1)\geq k,...,X_t(i_{\tilde{n}})\geq k,X_t(i_{\tilde{n}+1})\geq k+1,...,X_t(i_{\tilde{n}+\tilde{m}})\geq k+1, X_t(j_{1})\geq k,...,X_t(j_{\tilde{r}})\geq k,X_t(j_{\tilde{r}+1})\geq k+1,\right.\\
&\left....,X_t(j_{\tilde{r}+\tilde{s}})\geq k+1\right] - 
\mathbb{P}\left[X_t(i_1)\geq k,...,X_t(i_{\tilde{n}})\geq k,X(i_{\tilde{n}+1})\geq k+1,...,X_t(i_{\tilde{n}+\tilde{m}})\geq k+1\right]\mathbb{P}\left[X_t(j_{1})\geq k,\right.\\
&\left....,X_t(j_{\tilde{r}})\geq k,X_t(j_{\tilde{r}+1})\geq k+1,...,X_t(j_{\tilde{r}+\tilde{s}})\geq k+1\right],
\end{aligned}
\end{equation*}
where $\tilde{n}\leq n, \tilde{m}\leq m, \tilde{r}\leq r, \tilde{s}\leq s$.
Using the same strategy as in the proof of Proposition \ref{propagationofchaos}, we have that each of these terms are bounded by $\mathbb{P}\left[\left( \cup_{k=1}^{\tilde{n} + \tilde{m}} \psi_{i_k}\right)\cap \left( \cup_{h=1}^{\tilde{r}+\tilde{s}}\psi_{j_h}\right)\neq \emptyset \right]\leq D^2 \mathbb{P}\left[\psi_i \cap \psi_j \neq \emptyset\right]$ .
Hence
\[
\V (S(\pi^N_{k+1}(t),\pi^N_{k}(t) )) \leq C(D)\dfrac{N^{D+1}}{N-D} \left[ N \ln \left(\frac{N+u_N(t)-1}{N}\right) - \dfrac{(N-1)(u_N(t)-1)}{(N+u_N(t)-1)}\right],
\]
where $u_N(t)$ is defined in \eqref{udef} and $C(D)$ is a constant that depends only on $D$. We conclude that $\V (\lambda_t^{k,N})$ is of order $N^{-D+2}$ and it goes to zero as $N$ grows. 
\begin{flushright}
	$\square$	
\end{flushright}

\subsection{Convergence to stationarity for the limiting tagged queue}\label{appstat}

Define the Markov process $Z_t=\big( X_t, R_1(t),\ldots, R_{X_t}(t) \big)$ on $ \mathbb N \times \cup_{i \in\mathbb N} \mathbb R^i$ describing the dynamics of the limiting tagged queue (with a fixed service discipline) with Poisson arrivals with state dependent intensity  $\lambda_\infty^k$ when the system has $k$ customers.
% By Bramson, for fix $N$, $\lim\limits_{t \rightarrow \infty} \lambda^{N,k}_t = \lambda^{N,k}_\infty$, and by previous result $\lim\limits_{N \rightarrow \infty}\lambda^{N,k}_\infty= \lambda^{k}_\infty$

Positive recurrence of $Z$ for symmetric service policies follows directly from the insensitivity property of the limiting tagged queue (see \cite{Z}) since its stationary distribution is known in the case of exponentially distributed service times.

For FIFO, this is a more subtle question and it has been proved by \cite{B}.

%For symmetric scheduling, the results of \cite{Z} imply that $\limsup_x \lambda(x) < 1$ is a sufficient condition for positive recurrence of the process $Z$. The same result can de deduced from the results of \cite{AM} for FIFO scheduling. As a consequence, in both cases, there exists an almost surely finite time $T$ such that for any initial condition, the system is empty at time $T$.

\begin{proposition}
The Markov process $Z$ converges in total variation towards a stationary distribution $\rho$. If moreover, the service time distribution has an exponential moment, then there exist constants $c_1$ and $c_2$ such that for all initial state $z$
$$
|| P^z_t(\cdot) -\rho  ||_{TV} \le C_1(z) \exp(- C_2 t).
$$
\end{proposition}

\begin{proof}
We apply a continuous time version of Theorem 3.6 in \cite{Hairer}.  \\
Define the Lyapunov function:
$$ V((x,r))=  \textbf{1}_{\{x > c\} } \exp \left( \theta \sum_{i  \le x} r_i\right).$$

We first need to prove that level sets of $V$ are small, in the sense that
there exists $\alpha>0$ and $t$ such that for all initial conditions $z,z'$ which satisfy
$$V(z)+V(z') \le \kappa,$$
we have that
$$ || P^z_t(\cdot) -P^{z'}_t(\cdot)  ||_{TV} \le (1-\alpha).$$

Define $Z^1$ and $Z^2$ two processes with initial conditions $z$ and $z'$.
Let us consider the usual coupling where the two processes get identical arrival and service times. Note that if the processes meet, as a consequence of the coupling, they get the exact same distribution after this meeting time.
%Note that the  meeting time $\tau_{z,z'}$ is almost surely finite since the processes $Z^1$ and $Z^2$ are supposed to be (Harris) positive recurrent.

Hence, 
$$ || P^z_{T_\kappa}(\cdot) -P^{z'}_{T_\kappa}(\cdot)  ||_{TV} \le  P( \tau_{z,z'} > T_{\kappa} ) \le P(N_{T_{\kappa}} =0)<1,$$
for $T_{\kappa}$ the deterministic time corresponding to the maximum workload such that $V(z)+V(z') \le \kappa,$ and $N$ a Poisson process with bounded rate (since the arrival rate was proved to be bounded).
Observe that the event $\{N_{T_{\kappa}} =0 \}$ hence implies that there is no arrival before both queues empty (indeed, since the scheduling is work conserving, if the workload  is bounded by $T_\kappa$, then both queues are empty at time $T_\kappa$ in the absence of further arrivals).

%ine the workload $W=\sum_i r_i$. Note that levels sets for $V$ correspond to level sets for $W$ if $x>c$.
% Define $W_max$ the maximum of the two initial workloads.

% Let us perform the following coupling between two processes with different inital conditions.
% If $x \le x'$, both processes receive an intensity $\lambda(x')$ and the same service requirements while with intensity $\lambda(x)-\lambda(x')$, only the smallest process receives arrivals.
% Note that with this coupling, for all work conserving service disciplines, both processes are such that the difference in the number of jobs is at most 1 after an almost surely finite time.
% Hence each time one of the process vanishes, the probability that the other process gets also to zero is smaller than $P(S \ge Exp(\sup_x \lambda))=p<1$ 
% 
% 
% Note that if we can couple the processes such that they are both $0$ in finite time, then the claim is proved.
% A first observation is that after the hitting time of $0$ of the largest of the two processes, the number of tasks
% in the two processes has to be equal at a random finite time $\tau$.
% We can couple the processes after that time with the same arrivals (if they are in the same position)
% and with the following coupling otherwise:
% 
% which concludes the claim, given that all the random times involved are almost surely finite.

Then observe that, by Corollary \ref{limsuplambda}, we indeed have that 
$\limsup_{k \to \infty} \lambda_\infty^k = 0$ when $\lambda <1$.

Recall that the workload evolution does not depend on the scheduling. Hence, for any differentiable function $f$ that depends only on the workload $w$, the drift of $f((X_t,R_t))$ given $(x,r)$, can be computed as
$$
\LL f(w)= \Big({d \over dt}\E^{(x,r)}(f(X_t, R_t)\Big)_{t=0}= -f'(w)+ \lambda_\infty^x \E(f(w+S)-f(w)).
$$
 for all work conserving disciplines. 
In the case where $x$ is big enough, $V(x,r)$ is equal to $e^{\theta w}$, where $w=\sum_{i \le x} r_i$ is the workload. Hence, the drift of $V$ for $x$ sufficiently big is given by: 
\begin{equation*}
\begin{split}
\LL V((x,r))&= - \theta V((x,r)) + \lambda_\infty^x  (\E(e^{\theta S})-1) V((x,r))\\
&= \big(- \theta + (\E(e^{\theta S})-1)) \lambda_\infty^x \big)  V((x,r)), 
\end{split}
\end{equation*}
%for $x$ large enough.
% & \le  \theta (-1 + (1-\epsilon)(\E(S)+ \theta \tilde \epsilon_\theta) V((x,r), 
%where we first used a Taylor development for $\theta$ close to $0$ and also that 

Note that if $x>c$, with $c$ large enough,
then $\lambda_\infty^k  (\E(e^{\theta S})-1)) \le  1-\epsilon $.
Therefore, for $x$ large enough:
\begin{align*}
\LL V((x,r))&\le - \gamma V((x,r)). 
\end{align*}

This drift inequality is then sufficient to apply Theorem 3.6 in \cite{Hairer}. 	
\end{proof}

\newpage
\bibliographystyle{amsplain}

\end{document}